\documentclass[12pt,final]{amsart}
\usepackage{amsfonts}
\usepackage{mathrsfs}       
\usepackage{txfonts}
\usepackage{amssymb}
\usepackage{eucal}
\usepackage{graphicx}
\usepackage{amsmath}
\usepackage{amscd}
\usepackage[all]{xy}           
\usepackage{tikz}
\usepackage{amsfonts,latexsym}
\usepackage{xspace}
\usepackage{epsfig}
\usepackage{float}
\usepackage{color}
\usepackage{fancybox}
\usepackage{colordvi}
\usepackage{multicol}
\usepackage{colordvi}
\usepackage[active]{srcltx} 
\usepackage{ifpdf}
\ifpdf
\usepackage[colorlinks,final,backref=page,hyperindex]{hyperref}
\else
\usepackage[colorlinks,final,backref=page,hyperindex,hypertex]{hyperref}
\fi

\usepackage{wasysym}

\newcommand{\nc}{\newcommand}
\newcommand{\delete}[1]{}

	\nc{\mlabel}[1]{\label{#1}}  
	\nc{\mcite}[1]{\cite{#1}}  
	\nc{\mref}[1]{\ref{#1}}  
	\nc{\meqref}[1]{\eqref{#1}} 
	\nc{\mbibitem}[1]{\bibitem{#1}} 

\delete{
\nc{\mlabel}[1]{\label{#1}  
	{\hfill \hspace{1cm}{\bf{{\ }\hfill(#1)}}}}
\nc{\mcite}[1]{\cite{#1}{{\bf{{\ }(#1)}}}}  
\nc{\mref}[1]{\ref{#1}{{\bf{{\ }(#1)}}}}  
\nc{\meqref}[1]{\eqref{#1}{{\bf{{\ }(#1)}}}} 
\nc{\mbibitem}[1]{\bibitem[\bf #1]{#1}} 
}

\nc\delet{\delete}

\newtheorem{theorem}{Theorem}[section]
\newtheorem{prop}[theorem]{Proposition}
\newtheorem{lemma}[theorem]{Lemma}
\newtheorem{coro}[theorem]{Corollary}
\newtheorem{questions}[theorem]{Questions}
\theoremstyle{definition}
\newtheorem{defn}[theorem]{Definition}
\newtheorem{prop-def}{Proposition-Definition}[section]

\newtheorem{remark}[theorem]{Remark}

\newtheorem{tempex}[theorem]{Example}
\newtheorem{tempexs}[theorem]{Examples}
\newtheorem{temprmk}[theorem]{Remark}
\newtheorem{tempexer}{Exercise}[section]
\newenvironment{exam}{\begin{tempex}\rm}{\end{tempex}}

\nc{\vsa}{\vspace{-.1cm}} \nc{\vsb}{\vspace{-.2cm}}
\nc{\vsc}{\vspace{-.3cm}} \nc{\vsd}{\vspace{-.4cm}}
\nc{\vse}{\vspace{-.5cm}}

\nc{\name}[1]{{\bf #1}}

\nc{\NS}{U_{NS}}
\nc{\FN}{F_{\mathrm Nij}}
\nc{\dfgen}{V} \nc{\dfrel}{R}
\nc{\dfgenb}{\vec{v}} \nc{\dfrelb}{\vec{r}}
\nc{\dfgene}{v} \nc{\dfrele}{r}
\nc{\dfop}{\odot}
\nc{\dfoa}{\dfop^{(1)}} \nc{\dfob}{\dfop^{(2)}}
\nc{\dfoc}{\dfop^{(3)}} \nc{\dfod}{\dfop^{(4)}}
\nc{\mapm}[1]{\lfloor\!|{#1}|\!\rfloor}
\nc{\cmapm}[1]{\frakC(#1)}
\nc{\red}{\mathrm{Red}}
\nc{\cm}{C}
\nc{\supp}{\mathrm{Supp}}
\nc{\lex}{\mathrm{lex}}

\nc{\disp}[1]{\displaystyle{#1}}
\nc{\bin}[2]{ (_{\stackrel{\scs{#1}}{\scs{#2}}})}  
\nc{\binc}[2]{ \left (\!\! \begin{array}{c} \scs{#1}\\
    \scs{#2} \end{array}\!\! \right )}  
\nc{\bincc}[2]{  \left ( {\scs{#1} \atop
    \vspace{-.5cm}\scs{#2}} \right )}  
\nc{\sarray}[2]{\begin{array}{c}#1 \vspace{.1cm}\\ \hline
    \vspace{-.35cm} \\ #2 \end{array}}
\nc{\bs}{\bar{S}} \nc{\ep}{\epsilon}
\nc{\dbigcup}{\stackrel{\bullet}{\bigcup}}
\nc{\la}{\longrightarrow} \nc{\cprod}{\ast} \nc{\rar}{\rightarrow}
\nc{\dar}{\downarrow} \nc{\labeq}[1]{\stackrel{#1}{=}}
\nc{\dap}[1]{\downarrow \rlap{$\scriptstyle{#1}$}}
\nc{\uap}[1]{\uparrow \rlap{$\scriptstyle{#1}$}}
\nc{\defeq}{\stackrel{\rm def}{=}} \nc{\dis}[1]{\displaystyle{#1}}
\nc{\dotcup}{\ \displaystyle{\bigcup^\bullet}\ }
\nc{\sdotcup}{\tiny{ \displaystyle{\bigcup^\bullet}\ }}
\nc{\fe}{\'{e}}
\nc{\hcm}{\ \hat{,}\ } \nc{\hcirc}{\hat{\circ}}
\nc{\hts}{\hat{\shpr}} \nc{\lts}{\stackrel{\leftarrow}{\shpr}}
\nc{\denshpr}{\den{\shpr}}
\nc{\rts}{\stackrel{\rightarrow}{\shpr}} \nc{\lleft}{[}
\nc{\lright}{]} \nc{\uni}[1]{\tilde{#1}} \nc{\free}[1]{\bar{#1}}
\nc{\freea}[1]{\tilde{#1}} \nc{\freev}[1]{\hat{#1}}
\nc{\dt}[1]{\hat{#1}}
\nc{\wor}[1]{\check{#1}}
\nc{\intg}[1]{F_C(#1)}
\nc{\den}[1]{\check{#1}} \nc{\lrpa}{\wr} \nc{\mprod}{\pm}
\nc{\dprod}{\ast_P} \nc{\curlyl}{\left \{ \begin{array}{c} {} \\
{} \end{array}
    \right .  \!\!\!\!\!\!\!}
\nc{\curlyr}{ \!\!\!\!\!\!\!
    \left . \begin{array}{c} {} \\ {} \end{array}
    \right \} }
\nc{\longmid}{\left | \begin{array}{c} {} \\ {} \end{array}
    \right . \!\!\!\!\!\!\!}
\nc{\lin}{\call} \nc{\ot}{\otimes}
\nc{\ora}[1]{\stackrel{#1}{\rar}}
\nc{\ola}[1]{\stackrel{#1}{\la}}
\nc{\scs}[1]{\scriptstyle{#1}} \nc{\mrm}[1]{{\rm #1}}
\nc{\margin}[1]{\marginpar{\rm #1}}   
\nc{\dirlim}{\displaystyle{\lim_{\longrightarrow}}\,}
\nc{\invlim}{\displaystyle{\lim_{\longleftarrow}}\,}
\nc{\mvp}{\vspace{0.5cm}}
\nc{\mult}{m}       
\nc{\svp}{\vspace{2cm}} \nc{\vp}{\vspace{8cm}}
\nc{\proofbegin}{\noindent{\bf Proof: }}
\nc{\proofend}{$\blacksquare$ \vspace{0.5cm}}
\nc{\sha}{{\mbox{\cyr X}}}  
\nc{\ncsha}{{\mbox{\cyr X}^{\mathrm NC}}}
\newfont{\scyr}{wncyr10 scaled 550}
\nc{\ssha}{\mbox{\bf \scyr X}}
\nc{\ncshao}{{\mbox{\cyr
X}^{\mathrm NC,\,0}}}
\nc{\shpr}{\diamond}    
\nc{\shprc}{\shpr_c}
\nc{\shpro}{\diamond^0}    
\nc{\shpru}{\check{\diamond}} \nc{\spr}{\cdot}
\nc{\catpr}{\diamond_l} \nc{\rcatpr}{\diamond_r}
\nc{\lapr}{\diamond_a} \nc{\lepr}{\diamond_e} \nc{\sprod}{\bullet}
\nc{\un}{u}                 
\nc{\vep}{\varepsilon} \nc{\labs}{\mid\!} \nc{\rabs}{\!\mid}
\nc{\hsha}{\widehat{\sha}} \nc{\psha}{\sha^{+}} \nc{\tsha}{\tilde{\sha}}
\nc{\lsha}{\stackrel{\leftarrow}{\sha}}
\nc{\rsha}{\stackrel{\rightarrow}{\sha}} \nc{\lc}{\lfloor}
\nc{\rc}{\rfloor} \nc{\sqmon}[1]{\langle #1\rangle}
\nc{\altx}{\Lambda} \nc{\vecT}{\vec{T}} \nc{\piword}{{\mathfrak P}}
\nc{\lbar}[1]{\overline{#1}}
\nc{\dep}{\mathrm{dep}}

\nc{\mmbox}[1]{\mbox{\ #1\ }}
\nc{\ayb}{\mrm{AYB}} \nc{\mayb}{\mrm{mAYB}} \nc{\cyb}{\mrm{cyb}}
\nc{\ann}{\mrm{ann}} \nc{\Aut}{\mrm{Aut}} \nc{\cabqr}{\mrm{CABQR
}} \nc{\can}{\mrm{can}} \nc{\colim}{\mrm{colim}}
\nc{\Cont}{\mrm{Cont}} \nc{\rchar}{\mrm{char}}
\nc{\cok}{\mrm{coker}} \nc{\dtf}{{R-{\rm tf}}} \nc{\dtor}{{R-{\rm
tor}}}

\nc{\Div}{{\mrm Div}} \nc{\End}{\mrm{End}} \nc{\Ext}{\mrm{Ext}}
\nc{\FG}{\mrm{FG}} \nc{\Fil}{\mrm{Fil}} \nc{\Frob}{\mrm{Frob}}
\nc{\Gal}{\mrm{Gal}} \nc{\GL}{\mrm{GL}} \nc{\Hom}{\mrm{Hom}}
\nc{\hsr}{\mrm{H}} \nc{\hpol}{\mrm{HP}} \nc{\id}{\mrm{id}} \nc{\Id}{\mathrm{Id}}
\nc{\im}{\mrm{im}} \nc{\incl}{\mrm{incl}} \nc{\Loday}{\mrm{ABQR}\
} \nc{\length}{\mrm{length}} \nc{\LR}{\mrm{LR}} \nc{\mchar}{\rm
char} \nc{\pmchar}{\partial\mchar} \nc{\map}{\mrm{Map}}
\nc{\MS}{\mrm{MS}} \nc{\OS}{\mrm{OS}} \nc{\NC}{\mrm{NC}}
\nc{\rba}{\rm{Rota-Baxter algebra}\xspace}
\nc{\rbas}{\rm{Rota-Baxter algebras}\xspace}
\nc{\rbw}{\rm{RBW}\xspace}
\nc{\rbws}{\rm{RBWs}\xspace}
\nc{\rbadj}{\rm{RB}\xspace}
\nc{\mpart}{\mrm{part}} \nc{\ql}{{\QQ_\ell}} \nc{\qp}{{\QQ_p}}
\nc{\rank}{\mrm{rank}} \nc{\rcot}{\mrm{cot}} \nc{\rdef}{\mrm{def}}
\nc{\rdiv}{{\rm div}} \nc{\rtf}{{\rm tf}} \nc{\rtor}{{\rm tor}}
\nc{\res}{\mrm{res}} \nc{\SL}{\mrm{SL}} \nc{\Spec}{\mrm{Spec}}
\nc{\tor}{\mrm{tor}} \nc{\Tr}{\mrm{Tr}}
\nc{\mtr}{\mrm{tr}}

\nc{\ab}{\mathbf{Ab}} \nc{\Alg}{\mathbf{Alg}}
\nc{\Bax}{\mathbf{CRB}} \nc{\Algo}{\mathbf{Alg}^0}
\nc{\cRB}{\mathbf{CRB}} \nc{\cRBo}{\mathbf{CRB}^0}
\nc{\RBo}{\mathbf{RB}^0} \nc{\BRB}{\mathbf{RB}}
\nc{\Dend}{\mathbf{DD}} \nc{\bfk}{{\bf k}} \nc{\bfone}{{\bf 1}}
\nc{\base}[1]{{a_{#1}}} \nc{\Cat}{\mathbf{Cat}}
 \nc{\DN}{\mathbf{DN}}
\nc{\NA}{\mathbf{NA}}
\nc{\SDN}{\mathbf{SDN}}
\nc{\Diff}{\mathbf{Diff}} \nc{\gap}{\marginpar{\bf
Incomplete}\noindent{\bf Incomplete!!}
    \svp}
\nc{\FMod}{\mathbf{FMod}} \nc{\Int}{\mathbf{Int}}
\nc{\Mon}{\mathbf{Mon}}
\nc{\RB}{\mathbf{RB}} \nc{\remarks}{\noindent{\bf Remarks: }}
\nc{\Rep}{\mathbf{Rep}} \nc{\Rings}{\mathbf{Rings}}
\nc{\Sets}{\mathbf{Sets}} \nc{\bfx}{\mathbf{x}}
\nc{\BA}{{\Bbb A}} \nc{\CC}{{\Bbb C}} \nc{\DD}{{\Bbb D}}
\nc{\EE}{{\Bbb E}} \nc{\FF}{{\Bbb F}} \nc{\GG}{{\Bbb G}}
\nc{\HH}{{\Bbb H}} \nc{\LL}{{\Bbb L}} \nc{\NN}{{\Bbb N}}
\nc{\QQ}{{\Bbb Q}} \nc{\RR}{{\Bbb R}} \nc{\TT}{{\Bbb T}}
\nc{\VV}{{\Bbb V}} \nc{\ZZ}{{\Bbb Z}}

\nc{\cala}{{\mathcal A}} \nc{\calb}{{\mathcal B}}
\nc{\calc}{{\mathcal C}}
\nc{\cald}{{\mathcal D}} \nc{\cale}{{\mathcal E}}
\nc{\calf}{{\mathcal F}} \nc{\calg}{{\mathcal G}}
\nc{\calh}{{\mathcal H}} \nc{\cali}{{\mathcal I}}
\nc{\calj}{{\mathcal J}} \nc{\calk}{{\mathcal K}}
\nc{\call}{{\mathcal L}}
\nc{\calm}{{\mathcal M}} \nc{\caln}{{\mathcal N}}
\nc{\calo}{{\mathcal O}} \nc{\calp}{{\mathcal P}}
\nc{\calr}{{\mathcal R}} \nc{\cals}{{\mathcal S}} \nc{\calt}{{\mathcal T}}
\nc{\calv}{{\mathcal V}} \nc{\calw}{{\mathcal W}} \nc{\calx}{{\mathcal X}}
\nc{\CA}{\mathcal{A}}

\nc{\frakA}{{\mathfrak A}}
\nc{\fraka}{{\mathfrak a}}
\nc{\frakB}{{\mathfrak B}}
\nc{\frakb}{{\mathfrak b}}
\nc{\frakC}{{\mathfrak C}}
\nc{\frakE}{{\mathfrak E}}
\nc{\frakJ}{{\mathfrak J}}
\nc{\frakK}{{\mathfrak K}}
\nc{\frakd}{{\mathfrak d}}
\nc{\frakF}{{\mathfrak F}}
\nc{\frakg}{{\mathfrak g}}
\nc{\frakm}{{\mathfrak m}}
\nc{\frakM}{{\mathfrak M}}
\nc{\frakMo}{{\mathfrak M}^0}
\nc{\frakP}{{\mathfrak P}}
\nc{\frakp}{{\mathfrak p}}
\nc{\frakS}{{\mathfrak S}}
\nc{\frakU}{{\mathfrak U}}
\nc{\frakSo}{{\mathfrak S}^0}
\nc{\fraks}{{\mathfrak s}}
\nc{\os}{\overline{\fraks}}
\nc{\frakT}{{\mathfrak T}}
\nc{\frakTo}{{\mathfrak T}^0}
\nc{\oT}{\overline{T}}
\nc{\frakX}{{\mathfrak X}}
\nc{\frakXo}{{\mathfrak X}^0}
\nc{\frakx}{{\mathbf x}}
\nc{\frakTx}{\frakT}      
\nc{\frakTa}{\frakT^a}        
\nc{\frakTxo}{\frakTx^0}   
\nc{\caltao}{\calt^{a,0}}   
\nc{\ox}{\overline{\frakx}} \nc{\fraky}{{\mathfrak y}}
\nc{\frakz}{{\mathfrak z}} \nc{\oX}{\overline{X}} \font\cyr=wncyr10

\nc{\tred}[1]{\textcolor{red}{#1}} \nc{\tblue}[1]{\textcolor{blue}{#1}}
\nc{\tgreen}[1]{\textcolor{green}{#1}} \nc{\tpurple}[1]{\textcolor{purple}{#1}}
\nc{\tviolet}[1]{\textcolor{violet}{#1}}

\nc{\xing}[1]{\tpurple{Xing: #1}}
\nc{\shilong}[1]{\tpurple{Shilong: #1}}
\nc{\li}[1]{\tred{#1}}
\nc{\lir}[1]{\tred{Li: #1}}
\nc{\hu}[1]{\tblue{Huhu: #1}}
\nc{\Hu}[1]{\tblue{ #1}}
\nc{\markus}[1]{\tviolet{Markus: #1}}
\nc{\Markus}[1]{\tviolet{#1}}

\nc{\lint}{\cum} \nc{\vint}{\cum}
\nc{\E}{\cale} \nc{\J}{\calj} \nc{\D}{\partial} \nc{\inv}[1]{{#1}^{-1}} \nc{\bvp}[2]{\boxed{\begin{array}{l}#1 \\#2\end{array}}}
\nc{\spe}{\NN||X||} \nc{\vir}{\ZZ||X||} \nc{\ratvir}{\QQ||X||} \nc{\lspe}{\NN[[X]]} \nc{\lvir}{\ZZ[[X]]}
\nc{\vs}{\calv} \nc{\cvs}{{\rm \hat{\calv}}} \nc{\com}{\,\boxdot\, } \nc{\bcom}{\,\square\, }
\nc{\qd}{\bf} \nc{\conj}{\mathrm{conj}}  \nc{\ord}{\rm ord}
\nc{\nder}[2]{#1^{(#2)}} \nc{\dist}[1]{d(#1)} \nc{\divid}[2]{{#1}^{[#2]}} \nc{\der}[1]{#1'} \nc{\pp}{\Pi} \nc{\dd}{\D} \nc{\rinv}{\cvs^*}
\nc{\osum}{+}

\nc{\cum}{{\textstyle \varint}}
\nc{\jint}{\cum_{\!\!E}}
\nc{\jt}{\cum_{\mkern-5mu T}}
\nc\derk{{\partial}}
\nc\pik{{P_K}}\nc\e{\E_{{\rm DRB}}}
\nc{\ilvir}{\lvir'}
\nc{\aexp}{{e^X}} \nc\dck{{L}}
\nc\dlk{\D K^{-1}} \nc\ilk{K\cum_{\aexp}}
\nc\umdiff[2]{\derk_{K^{-1}}\left(#1\right)#2+ #1\derk_{K^{-1}}\left(#2\right)-\derk_{K^{-1}}(1)#1#2}

\nc{\cancum}{\cum_{\kern-1.5pt 0}^x}
\nc{\comcum}{\cum_{\kern-2.5pt E}}
\nc{\expcum}{\cum_{\!\!\aexp}}

\begin{document}

\title[Integro-differential rings on species and derived structures]
{Integro-differential rings on species and derived structures}

\author{Xing Gao}
\address{School of Mathematics and Statistics, Lanzhou University
Lanzhou, 730000, China;
Gansu Provincial Research Center for Basic Disciplines of Mathematics
and Statistics, Lanzhou, 730070, China
}
\email{gaoxing@lzu.edu.cn}

\author{Li Guo}
\address{
Department of Mathematics and Computer Science,
Rutgers University,
Newark, NJ 07102, USA}
\email{liguo@rutgers.edu}

\author{Markus Rosenkranz}
\address{RISC, Johannes Kepler University, 4040 Linz, Austria}
\email{m.rosenkranz@risc.uni-linz.ac.at}



\author{Huhu Zhang}
\address{School of Mathematics and Statistics
	Yulin University,
	Yulin, Shaanxi 719000, China}
\email{huhuzhang@yulinu.edu.cn}

\author{Shilong Zhang}
\address{College of Science, Northwest A\&F University, Yangling 712100, Shaanxi, China}
\email{shlz@nwafu.edu.cn}

\date{\today}

\begin{abstract}
  In the theory of species, differential as well as integral operators are known to arise in a
  natural way. In this paper, we shall prove that they precisely fit together in the algebraic
  framework of integro-differential rings, which are themselves an abstraction of classical calculus
  (incorporating its Fundamental Theorem). The results comprise (set) species as well as linear
  species. Localization of (set) species leads to the more general structure of modified
  integro-differential rings, previously employed in the algebraic treatment of Volterra integral
  equations. Furthermore, the ring homomorphism from species to power series via taking generating
  series is shown to be a (modified) integro-differential ring homomorphism. As an application, a
  topology and further algebraic operations are imported to virtual species from the general theory
  of integro-differential rings.
\end{abstract}

\makeatletter
\@namedef{subjclassname@2020}{\textup{2020} Mathematics Subject Classification}
\makeatother
\subjclass[2020]{
18M30, 
17B38, 
45J05, 
47G20, 
12H05, 
34M15 
}

\keywords{Species, integro-differential ring, differential Rota-Baxter ring, modified integro-differential ring, generating series}

\maketitle

\vspace{-1cm}

\tableofcontents

\vspace{-1cm}

\setcounter{section}{0}
\allowdisplaybreaks

\section{Introduction}

In this paper, we endow species (as we shall see later, this comprises: virtual set species, virtual
linear species, and localized virtual set species) with the structure of an integro-differential
ring or modified integro-differential ring. The well-known passage to the generating series will be
seen to constitute a homomorphism of the corresponding (modified) integro-differential
rings. This encapsulates a powerful tool for working with species, as it translates combinatorial
relations into differential and/or integral equations, which may subsequently be investigated by
algebraic and (possibly) also analytic means. To this end, new operations on species will also be
introduced from the general theory of integro-differential rings.

\subsection{Species and their operations}

The notion of species was introduced by A. Joyal~\mcite{J0} as a general and powerful tool to
formalize the idea of combinatorial structures, (primarily) with the aim of counting and classifying
them. Species are formulated abstractly as functors from the category $\mathbb{B}$ of finite sets
with bijections to the category $\mathbb{E}$ of finite sets with functions, thus providing a way to
handle combinatorial objects independently of the approach used to define them---such as sets,
algorithms, constructions, recursion, etc. Species theory was developed further as a fundamental
tool in combinatorics~\mcite{BLL} and several other areas of mathematics~\mcite{AM}, providing deep
insights into the nature of combinatorial objects and their relationships. As an important starting
point, the enumerative properties of a combinatorial species are faithfully reflected in its
generating function.

Species come in two main flavors, both of which shall be considered here: Using mere sets as labels,
one has the basic variety mentioned above, i.e. functors $\mathbb{B} \to \mathbb{E}$ that we shall
refer to as \name{set species} (otherwise called $\mathbb{B}$-species or just \name{species}). In
some important combinatorial applications, the label sets carry a linear order and relabellings are
required to respect the order, which leads to the category $\mathbb{L}$ of linear orders with
monotone functions; the corresponding functors $\mathbb{L} \to \mathbb{B}$ are then called
\name{linear species} (or $\mathbb{L}$-species). As detailed below, both varieties of species come
with natural algebraic operations, and these are preserved by their corresponding generating
functions. In particular, the ring structure associated with these species is respected: The map
sending a species to its exponential generating function is a homomorphism from the species ring to
the ring of formal power series.

\subsection{Integro-differential rings and modified integro-differential rings}

For studying combinatorial structures in more depth, a natural derivation as well as a family of
integral operators---the so-called Joyal integrals---have been set up on the ring of set species. We
show in this paper that the derivation together with the Joyal integral with respect to the analytic
exponential endows the ring of virtual species with the structure of an integro-differential ring
(and \emph{a fortiori} with the structure of a differential Rota-Baxter ring). The same is also true
for virtual linear species (which come with a canonical notion of integral), where this result is
less surprising, due to the close proximity of linear species to formal power series. Going back to set species, we obtain the more general structure of a modified integro-differential ring on
localized virtual species.

Before continuing, let us recall and fix some preliminary notions (we shall formally introduce them
in Definitions \mref{de:intdiff} and \mref{de:rrb} below). The central notion of \name{differential
  ring} originated from the algebraic study of differential equations, abstracted from analysis by
Joseph Fels Ritt \mcite{Ri2} and subsequently developed further by Ellis Kolchin \mcite{Ko} in the
last century, blossoming into a vast field of research ever since~\mcite{CH,PS}.  A differential
ring is thus a ring $R$ together with an additive operator $\D \colon R\to R$ that satisfies the
\name{Leibniz rule}
\begin{equation*}
  \D (xy) = \D (x) y + x \D (y) \quad\text{ for all }\, x,y\in R.
  \mlabel{eq:diff0}
\end{equation*}
Such an operator is called a \name{derivation}.

There is also an abstraction of the simple integral operator $f(x)\mapsto \cum_{\!a}^x \, f(t)\,dt$,
commonly called \name{Rota-Baxter operator} in an algebraic setting. Consequently, a
\name{Rota-Baxter ring} (of weight zero) is defined to be a ring $R$ equipped with an additive operator
$\cum\colon R \to R$ satisfying the \name{Rota-Baxter axiom}
\begin{equation}
  \cum (x) \cum (y) = \cum (x \cum (y)) + \cum (y \cum (x)) \quad\text{ for all }\, x,y\in R,
  \mlabel{eq:rba0}
\end{equation}
as a \name{pure integration by parts formula} (pure because it does not appeal to the
derivation). The notion of Rota-Baxter operator was introduced in the 1960 study of G. Baxter in
probability~\mcite{B} and has been studied during recent decades in numerous
areas---see~\mcite{CK,Gub,Ro1,Ro2} and the references therein.

Incorporating differential and integral operator together as in calculus~\mcite{GK}, a
\name{differential Rota-Baxter ring} is a triple $(R,\D,\cum)$ consisting of a differential ring
$(R,\D)$ and Rota-Baxter ring $(R,\cum)$ satisfying the First Fundamental Theorem of Calculus
$\D \cum = \id$.

For some applications in analysis, such as boundary problems~\cite{R,RBE, RR}, a tighter coupling
between derivation and Rota-Baxter operator is needed. To this end, the notion of
\name{integro-differential ring} was introduced: This is again a triple $(R, \D, \cum)$ as above,
but satisfying instead of~\meqref{eq:rba0} the stronger \name{integration by parts formula} or
\name{hybrid Rota-Baxter axiom}~\mcite{RR} in the form
\begin{equation*}
  \cum (x' \cum (y)) = x \cum (y) - \cum (xy) \quad\text{ for all }\,  x,y\in R,
\end{equation*}
where $x':= \D x$. Thus every integro-differential ring is a differential Rota-Baxter ring. There
have been many subsequent studies of their theories and
applications~\mcite{ACPRR2,ACPRR1,BCQ,GGR,GRR,RRTB,ZGK}.

\subsection{Algebraic structures on species}
As mentioned before, one aim of introducing the notions of integro-differential and differential
Rota-Baxter algebra was to provide an algebraic framework for differential and integro-differential
equations, along with their boundary problems (including initial value problems). It is then but
natural to expect that such equations can also be studied in other contexts. Indeed, the species
derivation along with the Joyal integrals have stimulated us to uncover an integro-differential algebraic structure for species.

It is known that the ring of species together with the derivation forms a differential ring, implying all the usual properties of differentiation~\cite[Exc.~2.5.10 and~2.6.19]{BLL}, along with
some powerful applications and expansions~\mcite{Lab,Lab4}. We will subsequently show that, equipped
with the Joyal integral with respect to the analytic exponential, the ring of species is in fact an
integro-differential ring. The integro-differential structure then allows us to obtain additional
topological and algebraic structures on species such as those obtained in~\mcite{GKR}, using the
general theory of integro-differential rings. Furthermore, the ring of localized set species can be
endowed with a modified integral-differential ring structure, which was introduced recently~\mcite{GGL2} extracting from Volterra integral operators with separable kernels. As mentioned earlier, the passage to the generating series turns out to be an
integro-differential homomorphism, which might be helpful by translating questions on
integro-differential species relations into questions on integro-differential equations for power series. In another work~\mcite{FGPXZ}, built on the combinatorial structure of free Rota-Baxter algebras, the species of Rota-Baxter algebra is constructed and is shown to be a twisted bialgebra.

Let us briefly outline the remaining of the paper. In Section~\mref{sec:idrsp}, we introduce the
integro-differential structure on the ring of virtual set and linear species (Theorems~\mref{thm:intd} and ~\mref{thm:idls}). Then we provide a
modified integro-differential ring structure on rings of localized virtual set species (Theorem~\mref{thm:dir1}). The ring
homomorphism from virtual species to power series (via taking the generating series) is shown to be
an integro-differential homomorphism (Proposition~\mref{prop:homo}). Under localization at a species, this homomorphism extends to a homomorphism of modified integro-differential rings (Proposition~\mref{pp:mintdh}).

In Section~\mref{s:app}, we apply general results from the theory of integro-differential rings to
obtain a topology from filtrations on virtual (set/linear) species and on localized virtual set
species, thus making these into a topological integro-differential ring (Theorem~\mref{thm:cont4}). Further applications are facilitated by introducing new operations, such as divided powers (Definition~\mref{defn:divid}) and exponentiation (Definition~\mref{defn:expon}) relative to a
base species. The functorial composition of virtual set species is shown to be equipotent with the composition defined
for integro-differential rings (Theorem~\mref{pp:funcomp}).

\noindent
{\bf Notation. }
We write $\ZZ$ and $\NN$ to denote, respectively, the set of all / nonnegative 
integers. All rings considered in this paper are assumed to be unitary and commutative.

\section{Integro-differential rings on species}\mlabel{sec:idrsp}

In this section, we will impose an integro-differential structure on the species ring. For notation
and terminology on species not defined in this paper, the readers are primarily referred
to~\mcite{BLL} and then further to~\mcite{AM}.

\subsection{Derivation and Rota-Baxter operators on set species}\mlabel{ss:species}
We first recall some basic definitions and results on species~\mcite{BLL}.
Denote by $U+V$ and $U\times V$ the disjoint union (confer~\cite[\S1.3]{BLL} and Remark~2 there) and
direct product of sets $U$ and $V$, respectively.
\begin{defn}
A \name{(set) species of structures} is a rule $F$ which produces
\begin{enumerate}
\item for each finite set $U$, a finite set $F[U]$ and
\item for each bijection $\sigma\colon U\to V$, a function $F[\sigma]\colon F[U] \to F[V]$.
\end{enumerate}
The functions $F[\sigma]$ should further satisfy the following functorial properties:
\begin{enumerate}
\item[(i)] For all bijections $\sigma\colon U\to V$ and $\tau\colon V\to W$,
\begin{align}
F[\sigma\circ \tau] = F[\sigma] \circ F[\tau]. \mlabel{eq:func1}
\end{align}
\item[(ii)] For the identity map $\id_U\colon U\to U$,
\begin{align}
F[\id_U] = \id_{F[U]}. \mlabel{eq:func2}
\end{align}
\end{enumerate}
As mentioned in the Introduction, this amounts to saying that we have a
functor~$F\colon \mathbb{B} \to \mathbb{E}$.
\end{defn}

An element $s\in F[U]$ is called a \name{structure of species} $F$ or briefly an
\name{$F$-structure} on $U$. The function $F[\sigma]$ is called the \name{transport} of
$F$-structures along $\sigma$. It follows immediately from~(\mref{eq:func1}) and
(\mref{eq:func2}) that each transport $F[\sigma]$ is necessarily a bijection.

\begin{exam} Let $U$ be a finite set. In each of the following cases, the transport of structures $F[\sigma]$ is obvious (see~\mcite{BLL} for details).

\begin{enumerate}
\item The species $\calg$ of graphs, defined by
  $\calg[U]:= \{\text{simple graphs with vertext set $U$ }\}$. Similarly, there are species of
  trees and species of rooted trees.

\item The species $E$ of sets, defined by $E[U]:= \{U\}$. For each finite set $U$, there is a unique
  $E$-structure, namely the set $U$ itself.

\item The species $E_n$ with $n\geq 1$, characteristic of sets of cardinality $n$, defined by
  \[ E_n[U] := \left\{ \begin{array}{cc} \{U\}, & \text{ if } |U| = n,\\
                         \emptyset,& \text{ otherwise.}  \end{array}\right. \] Here $\emptyset$ is
  the empty set. In particular, one writes $X:= E_1$, characteristic of singletons, and $1:=E_0$,
  characteristic of the empty set, that is,
  \[ 1[U] := \left\{ \begin{array}{cc} \{U\}, & \text{ if } U =\emptyset,\\
                       \emptyset,& \text{ otherwise.}  \end{array}\right. \]
\item The \name{empty species}, denoted by 0, is defined by $0[U] = \emptyset$ for all $U$.
\end{enumerate}
The species $1$ is the multiplicative identity, as will be made clear later.
\end{exam}

\begin{defn}
\mlabel{defn:iso}
  Let $F$ and $G$ be two species of structures. An \name{isomorphism from $F$ to $G$} is a family of
  bijections $\varphi_U\colon F[U] \to G[U]$ satisfying the following naturality condition: For any
  bijection $\sigma\colon U \to V$ between two finite sets, the following diagram commutes:
\[
\xymatrix{
F[U]\ar^{\varphi_U}[r] \ar_{F[\sigma]}[d]&G[U]\ar_{G[\sigma]}[d]\\
F[V]\ar^{\varphi_V}[r] &G[V]
}
\]
\end{defn}

In categorical terms, we have a natural isomorphism (i.e. an invertible natural
transformation). Since isomorphic species essentially possess the ``same'' combinatorial properties,
as usual they will henceforth be considered equal in the combinatorial algebra to be developed. Thus
we write $F = G$ to indicate that $F$ and $G$ are isomorphic. We next recall some operations on
species of structures.

\begin{defn}
  Let $F$ and $G$ be two species of structures. The species $F + G$, called the \name{sum} of $F$
  and $G$, is defined as follows: an $(F + G)$-structure on a finite set $U$ is an $F$-structure on
  $U$ or (exclusive) a $G$-structure on $U$. In other words, we have the disjoint union
  \begin{equation*}
    (F+G)[U]:= F[U] + G[U].
    \mlabel{eq:usum}
  \end{equation*}
  The transport along $\sigma\colon U\to V$ is defined by setting, for an $(F + G)$-structure $s$ on
  $U$,
  \begin{equation*}
    (F+G)[\sigma](s):= \left\{ \begin{array}{cc} F[\sigma](s) & \text{ if } s\in F[U],\\
    G[\sigma](s)& \text{ if } s\in G[U].  \end{array}\right.
    \mlabel{eq:tsum}
  \end{equation*}
\end{defn}

For a species $F$, a finite set $U$ and $n\geq 0$, denote by $nF$ and $nU$ the species $F+ \cdots +F$ ($n$ terms) and the disjoint union $U+\cdots +U$ ($n$ terms), respectively. Note that
$(nF)[U] = n \,F[U]$.

As expected, addition of species is associative and commutative, up to isomorphism, so that
$F+G = G+F\text{ and } (F+G)+H = F+(G+H)$. Moreover, the empty species $0$ is the neutral element
for addition, $F + 0 = 0+F = F$.

\begin{defn}
  A family $(F_i)_{i\in I}$ of species is said to be \name{summable} if, for any finite set $U$, one
  has $F_i[U] = \emptyset$ except for finitely many indices $i\in I$.  The sum of $(F_i)_{i\in I}$
  is then denoted by $\sum_{i\in I} F_i$.
\end{defn}

For example, the family $(E_n)_{n\geq 0}$ is summable with sum $E = E_0 + E_1 +E_2 +\cdots$. In
general, for any species $F$, we have~\cite[p.~30]{BLL}
\begin{equation}
F = F_0 + F_1 +F_2 + \cdots,
\mlabel{eq:cano}
\end{equation}
where $F_n$ is the species $F$ restricted to cardinality $n\geq 0$, meaning
\[
F_n[U] =\left\{\begin{array}{ll}F[U], & \quad \text{if $|U|=n$},\\ \emptyset, & \quad \text{otherwise. }\end{array} \right.
\]
The expression in ~(\mref{eq:cano}) is called the \name{canonical decomposition} of $F$.

\begin{defn}
Let $F$ and $G$ be two species of structures. The species $FG$, called the \name{product} of $F$ and $G$, is defined as follows:
for any finite set $U$
\begin{equation*}
(FG)[U]:= \sum_{V+W = U} F[V] \times G[W],
\mlabel{eq:prod}
\end{equation*}
where the disjoint sum is taken over all pairs $(V, W)$ forming a composition of $U$. The transport
along a bijection $\sigma\colon U \to V$ is carried out by setting, for each $(FG)$-structure
$s = (f, g)$ on $U$,
\begin{equation*}
  (FG)[\sigma](s):= (F[\varphi](f), G[\gamma](g)), \mlabel{eq:tprod}
\end{equation*}
where $\varphi:= \sigma|_V$ and $\gamma:= \sigma|_W$ are the restrictions of $\sigma$ to $V$ and
$W$, respectively.
\end{defn}

Informally, an $(FG)$-structure is an ordered pair consisting of an $F$-structure and a
$G$-structure over complementary disjoint subsets. The product of species is associative and
commutative up to isomorphism.
The product admits the species 1 as neutral element, and the species 0 as absorbing element, meaning
$1F=F1=F \ \text{ and }\ 0F=F0=0$.
Moreover, multiplication distributes over addition, $F(G+H) = FG + FH$. Under these operations, the
species form a semiring, denoted by $\spe$. The situation is analogous to the case of the semiring
$\NN$ of natural numbers, which lacks also a (global) operation of subtraction. The problem of
establishing a combinatorial form of subtraction has been solved by Joyal \mcite{J1, J2} and Yeh
\mcite{Y1, Y2} by the introduction of the concept of virtual species.

\begin{defn} Let $\spe$ be the semiring of species.
\begin{enumerate}
\item A \name{virtual species} is an element of the quotient set
$\vir := (\spe \times \spe)/ \sim,$
where the equivalence relation $\sim$ is defined by
$$(F, G) \sim (H, K) \Leftrightarrow F+K = G+H.$$
The class of $(F, G)$ according to $\sim$ is denoted by $F- G$.

\item Two species of structures $F$ and $G$ are said to be \name{unrelated} if the only
subspecies of $F$ which is isomorphic to a subspecies of $G$ is the empty species.
A virtual species $\Phi$ is said to be written in \name{reduced form} $\Phi = F - G$ if the species
$F$ and $G$ are unrelated.
\end{enumerate}
\mlabel{defn:virspec}

Every virtual species $\Phi$ can be written uniquely in reduced form
$\Phi= \Phi^+ - \Phi^-$~\cite[p.~122]{BLL}, where the species $\Phi^+$ (resp.~$\Phi^-$) are called
the \name{positive} (resp.~\name{negative}) part of $\Phi$.
\end{defn}

At this point, the ring of virtual species---loosely referred to earlier as the \name{species
  ring}---can be introduced in an explicit manner.

\begin{lemma}{\rm (\cite[p.~122]{BLL}, Yeh~\cite[Theorem~II.2.11]{Y2})}
The set $\vir$ of virtual species constitutes a commutative ring under
the operations of addition and multiplication defined by
\begin{eqnarray*}
 (F - G) + (H - K) &=& (F + H) - (G +K),\\
(F - G) (H - K)   &=& (FH + GK) - (FK + GH),
\end{eqnarray*}
with zero $0 = 0 - 0$ and unity $1 = 1-0$. The additive inverse of $F-G$ is $G-F$.

Moreover, the injection $\spe \hookrightarrow \vir, \ F \mapsto F-0$ is a semiring
homomorphism. Furthermore, $\vir$ is a factorial ring~\cite[Theorem~II.2.11]{Y2} and so, in
particular, an integral domain.  \mlabel{lem:virr}
\end{lemma}

The localization of the integral domain $\vir$ can constructed as usual~\cite{AMD}. Indeed, for any
multiplicative subset $S$ of $\vir$, one defines an equivalence relation $\sim$ on $\vir \times S$
by setting
$$(\Phi, s)\sim (\Psi, t) \quad\text{iff}\quad \Psi s - \Phi t = 0$$
for any $(\Phi, s), (\Psi, t)\in \vir \times S$. We write $S^{-1}\vir$ for the set of all
equivalence classes of $\sim$, and $\frac{\Phi}{s}$ for the equivalence class of $(\Phi, s)$. For
any $\frac{\Phi}{s}, \frac{\Psi}{t}\in S^{-1}\vir$, one defines
$$\frac{\Phi}{s}+\frac{\Psi}{t}:=\frac{\Phi t+\Psi s}{st} \qquad\text{and}\qquad \frac{\Phi}{s}\cdot
\frac{\Psi}{t}:=\frac{\Phi}{s} \frac{\Psi}{t}:=\frac{\Phi\Psi}{st} \, .$$
Then $( S^{-1}\vir, +, \cdot)$ is a commutative ring with identity $1:=\frac{1}{1}$, called the
\textbf{localization} of $\vir$ by $S$. We write $s^{-n}=\frac{1}{s^n}$ for any $n\in \mathbb{N}$.

\begin{remark}
  Regarding each element in $\mathbb{Z}$ as a species and taking $S = \mathbb{Z}\backslash \{0\}$,
  we write $\ratvir$ for the localization of $\vir$ by $S$; its elements are called \name{rational
    species} in \cite{Lab6}.  \mlabel{rem:rats}
\end{remark}

\begin{defn}
  Let $F$ be a species of structures. The species $\D F = F'$, called the \name{derivative} of $F$,
  is defined as follows: An $F'$-structure on $U$ is an $F$-structure on $U^+ := U \cup \{*\}$,
  where $*$ is a element chosen outside of $U$. In other words, one sets $F'[U]:= F[U^+]$.
The transport along a bijection $\sigma\colon U \to V$ is carried out by setting
$F'[\sigma] := F[\sigma^+]$ for any $F'$-structure $s$ on $U$,
where $\sigma^+\colon U + \{*\} \to V + \{*\}$ is the canonical extension with
$ \sigma^+(u):= \sigma(u) \text{ if } u\in U \text{ and } \sigma^+(*):= *$.
We refer to the operator $\D$ as the \name{combinatorial derivation}.\footnote{We prefer to use the
  term \name{combinatorial differential operators} for polynomials in $\D$, with coefficients coming
  from various domains such as $\NN$, $\ZZ$, $\spe$, $\vir$ and localizations.}
\smallskip

\noindent Figure~\mref{fig4} illustrates an $F'$-structure in a graphical way.
\end{defn}

\begin{figure}[h]
\begin{center}
\includegraphics[scale=0.6]{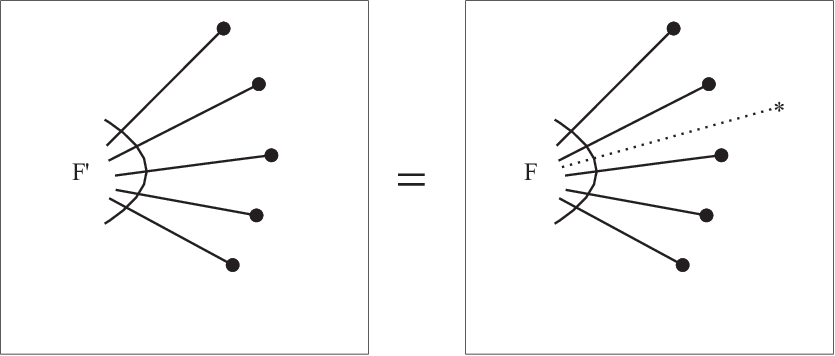}\caption{An $F'$-structure}\mlabel{fig4}
\end{center}
\end{figure}

Up to isomorphism, the derivation on species is additive and satisfies moreover the crucial Leibniz
rule~\cite[p.~52]{BLL}:
\begin{equation}
\mlabel{eq:spder}
(F+G)' = F'+G'\ \text{ and }\ (FG)' = F'G + FG'.
\end{equation}
Recall that~\cite[p.~125]{BLL} the derivation operator $\D$ defined on $\spe$ can be extended to
$\vir$ by additivity: For any virtual species $\Phi=\Phi^+ - \Phi^-$ in reduced form, one sets
$\Phi' = (\Phi^+)' - (\Phi^-)'$.  The resulting derivation $\D$ can then be extended further to
$\ratvir$. Since the properties~\meqref{eq:spder} are preserved, we have differential rings~$(\vir,
\D)$ and~$(\ratvir, \D)$ extending the differential semi\-ring~$(\spe, \D)$.

  We call a species $G$ \name{initialized} if $G[\emptyset] = \emptyset$ or,
equivalently, if $G_0 = 0$. If $F$ is another species, one may define~\cite[p.~41]{BLL} the
(partitional) \name{composition} or \name{substitution} $F\circ G$, also denoted by $F(G)$:
$$ (F\circ G)[X]:=\sum_{\pi \in \mathrm{Par}(X)} F[\pi]\times \prod_{p\in \pi} G[p],$$
where $\mathrm{Par}(X)$ denotes the set of partitions of $X$. 
The
singleton species $X$ is the neutral element for substitution so that $F = F\circ X = F(X)$ for all
species $F$. Moreover, the initialized species~$G$ are characterized by $G(0) = 0$, via substitution
of the empty species.

\begin{defn}
\begin{enumerate}
\item A \name{combinatorial differential constant} or briefly \name{differential constant} is a virtual species $K$ such that $K'= 0$. Elements in $\QQ$ are called \name{constants}.

\item  A \name{combinatorial differential tower} or briefly \name{differential tower} is a virtual species $T$ such that $T' = T$ and $T_0 = 1$.
\end{enumerate}\mlabel{defn:diffconst}
\end{defn}

Writing $\ker \dd$ for the set of all differential constants, we have $K\in \ker \dd$ if and only if
$K_n\in \ker \dd$ for all $n\geq 0$, where $K_n$ is the restriction of $K$ to cardinality $n$. It is
clear that we have $\QQ\subseteq \ker \dd$ in~$\ratvir$, with similar statements
for~$\NN \subseteq \spe$ and~$\ZZ \subseteq \vir$. There are plenty of \name{nonconstant
  differential constants}.
Indeed, one can show~\cite[p.125]{BLL} that for any $n \geq 0$, the species $X^n-n C_n$ is a
differential constant in~$\vir$, where $C_n$ is the $n$-cycle species. More generally, for any
nonnegative integers $k_i, m, n \geq 0$, we thus obtain
$$m+\sum_{i=1}^{n} k_i(X^i-i C_i)\in \ker \dd \subset \vir.$$

Differential towers are important tools for solving combinatorial differential equations.  There are
plenty of differential towers, the {most elementary} one being the \name{combinatorial exponential}
\begin{equation}
  \label{eq:comb-exp}
  E  = \sum_{n \ge 0} E_n = \sum_{n \ge 0} X^n/S_n \in \spe%
\end{equation}
and the most important for our subsequent development being the (canonical) \name{analytic   exponential}
\begin{equation}
  \label{eq:can-exp}
  e^X = \sum_{n \ge 0} X^n/n! \in \ratvir .
\end{equation}
Note that their generating series both yield the exponential function $E(x) = e^X(x) = e^x$.

\begin{remark}{\rm \mcite{Lab4}}
\begin{enumerate}
\item Every differential tower $T = T_0 + T_1 + \cdots+ T_n +\cdots$
is characterized by $T_0 = 1$ and $T_n'= T_{n-1}$ for all $n\geq 1$. \mlabel{it:d-1}

\item Every differential tower is of the form
$$T  = K E =(K_0  + K_1  + \cdots )(E_0  + E_1  +\cdots ),$$
where $K $ is a differential constant with $K_0=1$.  It follows that
\begin{equation}
T_{n}  = K_0  E_{n}  +  K_1  E_{n-1}  + \cdots +  K_n  E_{0}.
\mlabel{eq:decomt}
\end{equation}
Note that there is a bijective correspondence $\iota$ between differential towers $T$ and
differential constants $K$ with $K_0=1$. The map $\iota$ assigns to each $K$ the differential tower
$T = KE$. Its inverse is obtained by recursively solving ~(\mref{eq:decomt}) for $K_n E_0 =
K_n$. In the important special case $T = E$, the associated constant is of course $K = 1$.
\end{enumerate}
\mlabel{rem:cdtow}
\end{remark}

A virtual species $\Phi$ is called an \name{antiderivative} of a virtual species $\Psi$ if
$\Phi' = \Psi$.  There may be many antiderivatives of a virtual species~\cite[Theorem A]{Lab}.

\begin{defn}
  Let $T = T_0 + T_1 + \cdots+ T_n +\cdots $ be a differential tower. The virtual species
  \begin{equation}
    \jt \Phi:= T_1  \Phi  - T_2  \Phi'  + \cdots + (-1)^{n-1} T_n  \Phi^{(n-1)}  +\cdots = \sum_{i\geq 1} (-1)^{i-1}T_i \Phi^{(i-1)}
    \mlabel{eq:idt}
  \end{equation}
  is called the \name{Joyal integral} (with respect to $T$) of a virtual species $\Phi$. In
  particular, $\jint\Phi$ is called the \name{canonical Joyal integral} of $\Phi$.
\end{defn}

The following result from~\mcite{Lab4} shows that every antiderivative can be written as an
initialized Joyal integral plus a differential constant. (The case $T = E$ is treated in Proposition 13 of \cite[\S2.5]{BLL}.)

\begin{lemma}
  Let $T = T_0 + T_1 + \cdots+ T_n +\cdots $ be a differential tower. Then every antiderivative of a
  virtual species $\Phi$ can be written in the form
\begin{equation}
\Pi(\Phi) := K  + \jt \Phi,
\mlabel{eq:pi}
\end{equation}
where $K $ is a differential constant and $\jt \Phi$ is the Joyal integral~\meqref{eq:idt}.
Moreover, the latter is initialized so that $K(0) = \Pi(\Phi(0))$.  \mlabel{lem:allint}
\end{lemma}

\begin{defn}
  The {\bf generating series} of a species of structures $F$ is the formal power series
  $$F(x)=\sum_{n=0}^\infty f_n \frac{x^n}{n!} \in \QQ[[x]],$$
  where $f_n = |F[n]|$ is the number of $F$-structures on a set of $n$ elements.
\end{defn}

\begin{exam}\mlabel{ex:gs}
  Referring to the species operations described earlier, we have the generating series
  \begin{eqnarray*}
    X(x)&=&x\\
    (F+G)(x)&=&F(x)+G(x),\\
    (FG)(x)&=&F(x) \, G(x),\\
    (F\circ G)(x)&=&F(G(x)),\\
    F'(x)&=&\frac{d}{dx} \, F(x),
  \end{eqnarray*}
  to be found in Ex. 2(g) and Propositions~1.3.3(a), 1.3.8(a), 1.4.2(a), 1.4.8(a) of~\cite{BLL}.
\end{exam}

\begin{prop}\mlabel{prop:gsj}
  Let $T = T_0 + T_1 + \cdots+ T_n +\cdots $ be a differential tower and $\Phi$ a virtual
  species. Then the generating series of the Joyal integral is given by
  \begin{equation}
    (\jt \Phi)(x)= \sum_{i\geq 1} (-1)^{i-1}T_i(x) \, \Phi(x)^{(i-1)},
    \mlabel{eq:gsj}
  \end{equation}
  and we have $\frac{d}{dx}(\jt \Phi)(x)=\Phi(x)$.
\end{prop}
\begin{proof}
  This follows from ~\meqref{eq:idt} and Example~\mref{ex:gs}.
\end{proof}

\subsection{The integro-differential ring on set species}
\mlabel{ss:virtsp}

In this subsection, we establish an integro-differential ring structure on virtual species. Generalizing the algebraic abstraction of the derivation to the difference quotient whose
limit value is the derivation, a differential ring with weight was introduced in~\mcite{GK}. Rota-Baxter operators with weight originated from the probability study of G. Baxter
in 1960~\mcite{B}.

\begin{defn} \mlabel{de:intdiff} Let $\lambda\in \QQ$ be fixed.
  \begin{enumerate}
  \item A \name{differential ring of weight $\lambda$} is a ring $R$ together with an additive
    operator $\D \colon R\to R$ that satisfies $\D(1)=0$ and the \name{Leibniz rule}
    \begin{equation}
      \D xy = (\D x) y + x \D y + \lambda (\D x) (\D y) \quad\text{ for all }\, x,y\in R.
      \mlabel{eq:diffl}
    \end{equation}
    Such an operator is called a \name{derivation of weight $\lambda$} or a
    \name{$\lambda$-derivation}. We also use the common abbreviation $x' := \D x$.
  \item A \name{Rota-Baxter ring of weight $\lambda$} is a ring $R$ with an additive operator $\cum$
    on $R$ that satisfies the \name{Rota-Baxter axiom}
    \begin{equation}
      (\cum x) (\cum y) = \cum  x \cum y + \cum y \cum x   +\lambda \cum xy \quad\text{ for all }\, x,y\in R.
      \mlabel{eq:rba}
    \end{equation}
    Such an operator is called a \name{Rota-Baxter operator of weight $\lambda$} or a
    \name{$\lambda$-Rota-Baxter operator}.
  \item An \name{integro-differential ring of weight $\lambda$} is a differential ring $(R,\D)$ with
    an additive operator $\cum \colon R \to R$ that satisfies the \name{section axiom}
    $ \partial \cum = \id_R$ and the \name{hybrid Rota-Baxter axiom}
    \begin{equation}
      \cum x' \cum y = x \cum y - \cum xy - \lambda \cum x' y \quad\text{ for all }\,  x,y\in R.
      \mlabel{eq:ibpl}
    \end{equation}
    We call the two projections
    \begin{equation}
      \J:=\cum   \D \quad\text{and}\quad \E:=\id_{R}- \J
      \mlabel{eq:jd}
    \end{equation}
    the \name{initialization} and \name{evaluation}, respectively.
  \item A \name{differential Rota-Baxter ring of weight $\lambda$} is a differential ring $(R,\D)$
    of weight $\lambda$ with a Rota-Baxter operator $\cum$ of weight $\lambda$ such that
    $\D \cum=\id_R$.
  \end{enumerate}
\end{defn}

Here we have used \name{operator notation} for the derivation and Rota-Baxter operator, meaning
$\D x$ and $\cum x$ are shorthand for $\D(x)$ and $\cum(x)$, respectively.  We have moreover applied
the following convention for saving parentheses: Multiplication has precedence over integration, so
$\D xy$ and $\cum x\cum y$ are to be parsed as $\D(xy)$ and $\cum(x \cum y)$ , respectively. For
notational clarity, we shall also abbreviate the $i$-th derivative $\D^i x$ as $x^{(i)}$, for any
$i\geq 0$.

Note that we require that a derivation $\D$ satisfies $\D(1) = 0$, which automatically follows from
~(\mref{eq:diffl}) when $\lambda = 0$, but is a non-trivial restriction when $\lambda\neq
0$. Note that whenever the weight is not mentioned for the notions in Definition~\mref{de:intdiff}, it is
assumed to be zero.

Both the Rota-Baxter axiom~(\mref{eq:rba}) and the hybrid Rota-Baxter axiom~(\mref{eq:ibpl}) are
abstracted from integration by parts, but the latter is stronger than the former in the following
sense\cite[Proposition~9]{RRTB}.

\begin{lemma}
  Let $(R, \D, \cum)$ be an integro-differential ring of weight $\lambda$. Then $(R, \cum)$ is a
  Rota-Baxter ring of weight $\lambda$ and thus $(R,\D,\cum)$ a differential Rota-Baxter ring of
  weight $\lambda$.
  \mlabel{lem:intrb}
\end{lemma}

The following result gives equivalent characterizations of the hybrid Rota-Baxter axiom in
integro-differential rings~\cite[Theorem~2.5]{GRR}.

\begin{lemma}
  Let $(R,\D)$ be a differential ring of weight $\lambda$ with an additive operator $\cum: R \to R$
  such that $\D \cum = \id_R$.  Then with the notions in ~\meqref{eq:jd}, the following
  statements are equivalent (for $x, y \in R$):
  \begin{enumerate}
  \item We have an integro-differential ring $(R,\D,\cum)$.
  \item Evaluation is multiplicative, $\E(xy)= \E(x) \E(y)$. \mlabel{it:exy}
  \item Initialization is a weight $-1$ differential operator,
    $\J(x) \J(y) + \J(xy)= \J(x) y + x \J(y)$.
  \end{enumerate}
  \mlabel{lem:eqid}
\end{lemma}

Recall that in analysis, the classical antiderivative operators
$\cum_{\kern-2.5pt a}\colon C^\infty(\RR) \to C^\infty(\RR)$ are given by
$(\cum_{\kern-2.5pt a} \, f)(x) = \cum_{\kern-2pt a}^x \, f(\xi) \, d\xi$, parametrized by the
initialization point $a \in \RR$. In a similar way, the family of Joyal integrals ~\eqref{eq:idt}
is parametrized by the differential tower~$T$. For obtaining an integro-differential structure,
however, these differential towers turn out to be highly constrained because of the multiplicativity
requirement for the associated evaluation (see Theorem~\mref{thm:intd} below).

\begin{defn}
A species  $F=\sum_{n\geq 0} F_n$ is called an {\bf analytic exponential} if
\begin{equation}
	F_p F_q = {p+q \choose p} F_{p+q}
\mlabel{eq:binom}
\end{equation}
for all $p, q \geq 0$.
\mlabel{defn:binom}
\end{defn}

The term analytic exponential is justified because the following characterization shows that it
comprises only minor variations---essentially just a change of base---of the canonical
exponential~\meqref{eq:can-exp}. It is important to distinguish these analytic exponentials from the
combinatorial exponential~\meqref{eq:comb-exp}, as these species are \emph{not}
isomorphic~\cite[(1.32)]{Lab5}.

\begin{lemma}\mlabel{lem:exps}
  A species~$F$ is an analytic exponential if and only if $F = e^{F_1}$. Hence the only analytic exponential species are
  $e^{\lambda X}$ for $\lambda \in \ZZ \setminus \{ 0 \}$.
\end{lemma}
\begin{proof}
  We start by proving the equivalence, so assume first that $F=\sum_{n\geq 0} F_n$ is an analytic
  exponential. Then~\meqref{eq:binom} implies $nF_n=F_1F_{n-1}$ and thus $F_n=\tfrac{F_1^n}{n!}$ for
  any $n\geq 1$, which is equivalent to $F = e^{F_1}$. Conversely, $F_n=\tfrac{F_1^n}{n!}$ gives
  $$F_pF_q=\frac{F_1^p}{p!}\frac{F_1^q}{q!}=\frac{F_1^{p+q}}{p!q!}=\frac{(p+q)!}{p!q!}\frac{F_1^{p+q}}{(p+q)!}={p+q
    \choose p} F_{p+q}, \quad p, q\geq 0,$$
  showing that~$F$ is an analytic exponential.

  But $F_1$ is molecular in cardinality~$1$, so it must be a $\ZZ$-linear combination of molecular
  species in cardinality~$1$ by~\cite[Proposition~2.6.6]{BLL}. But from~\cite[\S2.6, (16)]{BLL} we know
  that the only molecular species in cardinality~$1$ is~$E_1 = X$ so that~$F_1 = \lambda X$ for some
  nonzero $\lambda \in \ZZ$.
\end{proof}

We note the following intuitive characterization of analytic exponentials, which is reminiscent but
distinct from the well-known exponential functional equation
\begin{equation*}
  \label{eq:1}
  F(X+Y) = F(X) \, F(Y),
\end{equation*}
which holds \emph{both} for the combinatorial exponential~\meqref{eq:comb-exp} and for the canonical
analytic exponential~\meqref{eq:can-exp}, as noted in~\cite[(2.16)]{Lab5}.

\begin{prop}
  A species $F$ is an analytic exponential  if and only if 
  \begin{equation*}
    F[U+V] = F[U] \times F[V],
    \mlabel{eq:sumprset}
  \end{equation*}
  for all finite sets $U$ and $V$.
  \mlabel{prop:eqbin}
\end{prop}

\begin{proof}
Denote $n=|U+V|$, $p=|U|$ and $q=|V|$. Notice that  for some nonzero~$\lambda \in \ZZ$, we have
$$e^{\lambda X}[U+V]= \frac{\lambda^n X^n}{n!}[U+V]
            = \frac{\lambda^n}{n!}\Big(\sum_{\ast_1+\cdots+\ast_n=U+V\atop \ast_i\neq\emptyset}X[\ast_1]\times \cdots\times X[\ast_n]\Big)
            = \frac{\lambda^n}{n!}\Big(n![\ast_1]^{n}\Big)=\lambda^n [\ast_1]^{n},$$
and so
$$e^{\lambda X}[U]=\lambda^p [\ast_1]^{p}, \quad e^{\lambda X}[V]=\lambda^q [\ast_1]^{q}.$$

Suppose $F$ is an analytic exponential.   It follows from   Lemma~\mref{lem:exps} that
$$F[U+V]=e^{\lambda X}[U+V]=\lambda^n [\ast_1]^{n}=\lambda^p [\ast_1]^{p}\times\lambda^q [\ast_1]^{q}=e^{\lambda X}[U]\times e^{\lambda X}[V]=F[U]\times F[V].$$

Conversely, by Lemma~\mref{lem:exps}, it suffices to prove the analytic exponential relation $F_n = F_1^n/n!$
for all $n \ge 0$. We work by induction on $n$.  For $U=V=\emptyset$, we have
$$F_0[\emptyset]=F[U+V]=F[U]\times F[V]=F_0[\emptyset]\times F_0[\emptyset]$$ and so
$F_0[\emptyset]=\{\emptyset\}.$ Thus $F_0=1$, that is, the result holds for $n=0$. We next show that
the analytic exponential relation for $n-1$ implies the relation for $n$.  Let $U+V$ be a set with $|U|=p$ and $|V|=q$.
By the induction hypothesis, we have
{\small\begin{eqnarray*}
&&F_{n}[U+V] =F[U+V]=F[U]\times F[V] =F_{p}[U]\times F_{q}[V]\\
&=&\frac{F_1^{{p}}}{{p}!}[U]\times \frac{F_1^{{q}}}{{q}!}[V]\\
&=&\Big(\sum_{\ast_1+\ast_2+\cdots+\ast_{{p}}=U\atop
\ast_i\neq\emptyset}\frac{1}{{p}!}F_1[\ast_1]\times
F_1[\ast_2]\times\cdots\times
F_1[\ast_{{p}}]\Big)
\times\Big(\sum_{\ast_1+\ast_2+\cdots+\ast_{{q}}=V\atop \ast_i\neq\emptyset}\frac{1}{{q}!}F_1[\ast_1]\times F_1[\ast_2]\times\cdots\times F_1[\ast_{{q}}]\Big)\\
&=& F_1[\ast_1]^{{n}}\\
&=& \frac{1}{{n}!}\Big(\sum_{\ast_1+\ast_2+\cdots+\ast_{{n}}=U+V\atop \ast_i\neq\emptyset}F_1[\ast_1]\times F_1[\ast_2]\times\cdots\times F_1[\ast_{{n}}]\Big)\\
&=& \frac{F_1^{{n}}}{{n}!}[U+V],
\end{eqnarray*}}
as required.
\end{proof}

We shall now show that the {\em only} differential tower $T$ yielding an integro-differential
structure on~$(\ratvir, \D)$ is the canonical analytic exponential $T = e^X$, whereas the
combinatorial exponential $T = E$ associated to the canonical Joyal integral does \emph{not} lead to
an integro-differential ring. As a consequence, the differential constant associated to the
``analytic integral operator'' is not $K=1$ as for the combinatorial exponential, but
$K = e^{X-\hat{X}}$ with $\hat{X}$ being the \name{species of pseudo-singletons} whose
characteristic property~\cite[(2.7)]{Lab5} is the combinatorial-analytic linking relation
$E = e^{\hat{X}}$.

\begin{theorem}
Let $T$ be a differential tower. The triple $(\ratvir, \D, \jt)$ is an integro-differential ring if and only if $T=e^X$.
The corresponding evaluation is given by
  \begin{equation}
    \E(\Phi) = \sum_{n\geq 0} \frac{(-X)^{n}}{n!} \, \Phi^{(n)}
    \mlabel{eq:eval}
  \end{equation}
  for all $\Phi \in \ratvir$.
  \mlabel{thm:intd}
\end{theorem}

\begin{proof}
First by ~\meqref{eq:spder}, the pair $(\ratvir,\D)$ is a differential ring.

We next show that for arbitrary differential
  towers~$T$, we obtain an integro-differential ring $(\ratvir, \D, \jt)$
  if and only if ~$T$ is exponential. So assume first that~$T$ is an analytic exponential so that~$T(X) = e^{\lambda
    X}$ for some nonzero~$\lambda \in \QQ$. Since~$T$ is a differential tower we must have~$T_1' =
  (\lambda X)' = \lambda = 1$, we have indeed the analytic exponential~$T(X) =
  e^X$. The section axiom $\D  \jt =
  \id_{\ratvir}$ follows from Lemma~\mref{lem:allint}.
 Further, by Lemma~\mref{lem:eqid}, we only need to show that
\begin{equation*}\mlabel{eq:hom}
 \E(\Phi\Psi)  = \E(\Phi) \E(\Psi) \quad \text{ for all }\, \Phi, \Psi\in \ratvir.
\end{equation*}
Note from ~(\mref{eq:idt}) that
\begin{equation*}
\E(\Phi) = \Phi   - \jt \Phi' = \Phi   - \sum_{n\geq 1} (-1)^{n-1}T_n \Phi^{(n)} = \Phi   + \sum_{n\geq 1} (-1)^{n}T_n \Phi^{(n)},
\end{equation*}
which immediately yields~\meqref{eq:eval}. Hence
{\small
\begin{align*}
  \E(\Phi\Psi) =&~\sum_{n\geq 0} (-1)^{n}T_n (\Phi\Psi)^{(n)} = \sum_{n\geq 0} (-1)^{n}T_n \sum_{p+q=n\atop p, q\geq 0} {n\choose p} \, \Phi^{(p)} \Psi^{(q)} \\
=&~\sum_{n\geq 0} (-1)^{n} \sum_{p+q=n\atop p, q\geq 0} {p+q \choose p} T_n\Phi^{(p)} \Psi^{(q)}
\overset{\meqref{eq:binom}}= \sum_{n\geq 0} (-1)^{n} \sum_{p+q=n\atop p, q\geq 0}T_{p}T_{q}\Phi^{(p)} \Psi^{(q)}\\
=&~ \sum_{p,q\geq 0}  (-1)^{p+q} T_{p}T_{q}\Phi^{(p)} \Psi^{(q)}
= \bigg(\sum_{p\geq 0}(-1)^{p} T_{p} \Phi^{(p)} \bigg) \bigg( \sum_{q\geq 0}  (-1)^{q}T_{q}\Psi^{(q)} \bigg) = \E(\Phi) \E(\Psi).
\end{align*}
}
For the converse, assume now~$(\ratvir, \D,
\jt)$ is an integro-differential ring. Then its evaluation must be multiplicative, so for any $p,
q\geq 0$ the expression
{\small
\begin{eqnarray*}
\E(X^{p}X^{q})&=&\E(X^{p+q}) = \sum_{k\geq 0} (-1)^{k}T_k (X^{p+q})^{(k)}
= \sum_{0\leq k < p+q} (-1)^{k}T_k (X^{p+q})^{(k)} +(-1)^{p+q}(p+q)! \, T_{p+q}\\
&=&\sum_{0\leq k < p+q} (-1)^{k}T_k \sum_{m+n=k\atop m,n\geq0}{m+n\choose
    m}(X^{p})^{(m)}(X^{q})^{(n)} +(-1)^{p+q} (p+q)! \, T_{p+q}\\
&=&\sum_{0\leq m \leq p, 0\leq n\leq q\atop m+n < p+q} {m+n\choose m} \, (-1)^{m+n} \, T_{m+n} \, (X^{p})^{(m)}(X^{q})^{(n)} +(-1)^{p+q} (p+q)! \, T_{p+q}
\end{eqnarray*}
}
must coincide with
\begin{eqnarray*}
  \E(X^{p})\E(X^{q}) &=& \sum_{0\leq m \leq p\atop 0\leq n \leq q}
                         (-1)^{m+n} \, T_mT_n \, (X^{p})^{(m)} \, (X^{q})^{(n)}\\
                     &=& \sum_{0\leq m \leq p, 0\leq n\leq q\atop m+n < p+q}
                         (-1)^{m+n} \, T_mT_n \, (X^{p})^{(m)} \, (X^{q})^{(n)}+(-1)^{p+q} p! q! \, T_{p}T_{q},
\end{eqnarray*}
and this yields---by the uniqueness of molecular decomposition---the ``leading-term'' equality
$$T_{p}T_{q} =T_{p+q} \frac{(p+q)!}{p! q!} ={p + q \choose p} \, T_{p+q},$$
which shows that the differential tower~$T$ is an analytic exponential and thus $T(X) = e^X$.
\end{proof}

Recall that $(\QQ[[x]], \frac{d}{dx}, \cancum)$ is an integro-differential ring, where
$\frac{d}{dx}$ and $\cancum$ is the usual derivation and Rota-Baxter operator,
respectively~\cite{GRR}. These are the operations on formal power series that correspond to the
analogous species operations, under the map sending each species to its generating series.

\begin{prop}\mlabel{prop:homo}
  The map
  $$(\ratvir, \D, \comcum) \rightarrow \Big(\QQ[[x]], \frac{d}{dx}, \cancum\Big), \quad \Phi\mapsto
  \Phi(x)$$ is a homomorphism of integro-differential rings.
\end{prop}
\begin{proof}
  We know already from Example~\mref{ex:gs} that the map is a differential ring homomorphism, so it
  remains to prove that it also preserves the Rota-Baxter structure. We have
  \begin{eqnarray*}
    (\comcum \Phi)(x)
    &\overset{\meqref{eq:gsj}}=& \sum_{i\geq 1} (-1)^{i-1} \big(e^X\big)_i(x) \, \Phi^{(i-1)}(x)
    = \sum_{i\geq 0} (-1)^i \frac{x^{i+1}}{(i+1)!} \, \Phi^{(i)}(x),\\
    &=& \sum_{i \ge 0} (-1)^i \frac{x^{i+1}}{(i+1)!} \sum_{j \ge i} \frac{x^{j-i}}{(j-i)!} \, \Phi^{(j)}(0),
  \end{eqnarray*}
  using in the last step the expansion
  $\Phi^{(i)}(x) = \sum_{j \ge i} \tfrac{x^{j-i}}{i! (j-i)!} \, \Phi^{(j)}(0)$ obtained from
  differentiating the Taylor development
  $\Phi(x) = \sum_{j \ge 0} \tfrac{x^j}{j!} \, \Phi^{(j)}(0)$. Interchanging the summations yields
  \begin{eqnarray*}
    (\comcum \Phi)(x)
    &=& \sum_{j \ge 0} \frac{x^{j+1}}{(j+1)!} \, \Phi^{(j)}(0) \sum_{i=0}^j (-1)^i \frac{(j+1)!}{(i+1)!
        \, (j-i)!} = \sum_{j \ge 0} \frac{x^{j+1}}{(j+1)!} \, \Phi^{(j)}(0) = \cancum \Phi(x),
  \end{eqnarray*}
  where the inner sum in the last but one step is computed by the binomial evaluation
  $$\sum_{i=0}^j \binom{j+1}{i+1} \, (-1)^i = 1 - \sum_{i=0}^{j+1} \binom{j+1}{i} \, (-1)^i =
  1 - (1-1)^{j+1} = 1,$$
  while the last step is plain Taylor development.
\end{proof}

\begin{coro} \mlabel{coro:RBV} The pair $(\ratvir,\D, \expcum)$ is a differential Rota-Baxter ring.
\end{coro}
\begin{proof}
  This follows from Lemma~\mref{lem:intrb} and Theorem~\mref{thm:intd}.
\end{proof}

Before concluding this subsection by equipping species with a matching Rota-Baxter structure, let us
briefly investigate the analytic integral operator~$\cancum$ and compare it with its combinatorial
counterpart~$\comcum$. From a combinatorial point of view, the latter is clearly advantageous since
it remains within the confines of (virtual) species, avoiding any extraneous rational
coefficients. But it would yield some unnatural integrals, at least when judged from the
perspective of calculus.

\begin{exam}\quad
  \begin{enumerate}
  \item\label{it:powint} For the \name{powers} one
    has~$\expcum \tfrac{X^k}{k!} = \tfrac{X^{k+1}}{(k+1)!}$ for all~$k \ge 0$. This follows
    from~$\expcum X^k = \tfrac{X^{k+1}}{k+1}$, which one shows by computing
    \begin{equation*}
      \jt\, X^k = \sum_{n \ge 0} (-1)^n \frac{X^{n+1}}{(n+1)!} \, k^{\underline{n}} \, X^{k-n}
      = \left( \sum_{n=0}^k \binom{k}{n} \, \frac{(-1)^n}{n+1} \right) \, X^{k+1},
    \end{equation*}
    where the parenthesized coefficient comes out as $\tfrac{1}{k+1}$ as follows: Integrating the
    binomial formula $$(x+1)^k = \sum_n \binom{k}{n} \, x^n$$ with $\cancum$ yields
    $$\tfrac{(x+1)^{k+1}}{k+1} - \tfrac{1}{k+1} = \sum_n \binom{k}{n} \, \tfrac{x^{n+1}}{n+1},$$ and
    evaluating this at $x=-1$ leads to $$-\tfrac{1}{k+1} = \sum_n \binom{k}{n} \,
    \tfrac{(-1)^{n+1}}{n+1},$$
    as required. In contrast, the combinatorial integral~$\comcum X^k$
    yields a different result (which seems hard to interpret).
Note that this implies $$\expcum L = \sum_{k \ge 1} \tfrac{X^k}{k} \neq \sum_{k \ge 1}
    \tfrac{X^k}{C_k} = \mathcal{C}$$
    even though it is known~\cite[(1.4.28)]{BLL} that the cycle species~$\mathcal{C}$ has derivative
    $\mathcal{C}' = L$. Obviously, the integration constant does not match up.

\item For the \name{analytic exponentials}, this
    implies~$$\expcum e^{\lambda X} = \lambda^{-1} \, (e^{\lambda X} - 1) = \tfrac{e^{\lambda
        X}}{\lambda} \, \big|_0^X,$$ as to be expected from the corresponding relation in
    calculus. In contrast, the \name{combinatorial exponential} yields $\expcum E = (e^{-X} - 1) E$
    as one checks immediately. On the other hand, the combinatorial integral operator gives
    $\comcum E = E_{\text{alt}} E$ and~$\comcum e^X = E_{\text{alt}} \, e^X$, where we use the
    alternating series
    $$E_{\text{alt}} := E_1 - E_2 + E_3 - E_4 +- \cdots = E_{\text{odd}} - E_{\text{even}} + 1.$$

  \item\label{it:nozeroin} The computation in Item~\ref{it:powint} instills some hope that the
    evaluation~\meqref{eq:eval} associated to the \name{analytic integral operator}~$\expcum$ might
    just be \name{initialization at the zero species?} In other words, we would have the
    evaluation~$\E_0(\Phi) := \Phi(0)$, tantamount to restricting a species to the empty label
    set. However, this is not the case as one may see by choosing a simple example such as the
    $3$-cycle species~$\Phi = C_3$ from Table~3 of \S2.6 in~\cite{BLL}. Since we have
    $\D(C_3) = X^2$ by Table 5 in Appendix 2 of~\cite{BLL}, we obtain $\E(\Phi) = C_3 - X^3/3$, and
    this is clearly nonzero because $\vir$ is a polynomial ring in the indeterminates
    $X, E_2, E_3, C_3, \dots$ according~\cite[Proposition~2.6.15]{BLL}.
Of course, we are ``not far away'' since the difference
    $$\E(\Phi) - \E_0(\Phi) = \E(\Phi) = C_3 - X^3/3$$ is clearly a differential constant (all
    evaluations map to the ring of differential constants).
  \item Perhaps the \name{combinatorial integral operator}~$\comcum$ has for its evaluation the
    \name{initialization at the zero species?} Again, the answer is no: Taking for
    example~$\Phi = X^2$, we obtain the ``evaluation''~$\Phi - \comcum \, \D \Phi = 2 E_2 - X^2$,
    which is of course different from~$\E_0(\Phi) = 0$. Note that we have used scare quotes since we
    know from the proof of Theorem~\mref{thm:intd} that such an ``evaluation'' would fail to be
    multiplicative and thus hardly deserves that name. (Incidentally, this is another proof that the
    ``evaluation'' cannot be initialization at the zero species, which is of course multiplicative.)
  \end{enumerate}
  \mlabel{ex:expcum}
\end{exam}

While this clarifies some important issues, some crucial questions do remain open at this
point. Among these, let us just mention the following two.

\begin{questions}\quad
  \begin{enumerate}
  \item We know already that any Joyal integral operator different from~$\expcum$ fails to yield an
    integro-differential structure on~$\vir$. But do we at least get a differential Rota-Baxter
    structure for (some) other differential towers?
  \item Can we characterize the evaluation~\meqref{eq:eval} associated to~$\expcum$ in some more
    natural way, like initialization at some quasi-zero species?
  \end{enumerate}
\end{questions}

We will now see how the analytic integral operator can be modified into a whole family of
Rota-Baxter operators that are matched up amongst themselves in the following precise
way~\cite{ZGG}. This result will shed some light on the unit-modified differential Reynolds ring
structure to be established in Section~\ref{ss:locsp}.

\begin{defn}
  A \name{matching Rota-Baxter ring} $\big(A, (P_\omega)_{\omega\in\Omega}\big)$ of weight zero is a
  ring $A$ equipped with a family of additive operators $ (P_\omega)_{\omega\in\Omega}$ on $A$
  satisfying
  $$P_\alpha(a) \, P_\beta(b)=P_\alpha(aP_\beta(b))+P_\beta(P_\alpha(a)b)$$
  for all~$\alpha, \beta \in \Omega$.
\end{defn}

\begin{theorem}
Let $\Omega$ be a nonempty  set of differential constants.
For any $\omega\in\Omega$, define additive operators
\begin{equation}
\begin{split}
\derk_{\omega}\colon \ratvir \to& \ratvir,\quad  \Phi \mapsto ({\omega}\Phi)'={\omega}\Phi',\\
P_{\omega}\colon\ratvir \to& \ratvir,\quad  \Phi \mapsto \expcum {\omega}\Phi ={\omega}\expcum \Phi.
\end{split}
\mlabel{eq:sder}
\end{equation}
Then
\begin{enumerate}
\item For any $\omega\in\Omega$, the operator $\derk_{\omega}$ is a differential operator on $\ratvir$$:$
$$\derk_{\omega}(\Phi \Psi)=\derk_{\omega}(\Phi)\Psi+ \Phi \derk_{\omega}(\Psi)\,\text{ for all }\, \Phi, \Psi\in\ratvir,$$
and $\derk_{\omega} \expcum={\omega} \, \id_{\ratvir}$.
\mlabel{it:dpk1}

\item The pair $\big(\ratvir, (P_{\omega})_{\omega\in\Omega}\big)$ is a matching Rota-Baxter ring.
\mlabel{it:dpk2}
\end{enumerate}
\mlabel{thm:dir0}
\end{theorem}

\begin{proof}
Let $\Phi, \Psi\in\ratvir$.

\mref{it:dpk1}
By ~\eqref{eq:sder}, for any $\omega\in\Omega$, we have
\begin{eqnarray*}
\derk_{\omega}(\Phi \Psi)={\omega} \Phi'\Psi+{\omega}\Phi\Psi'
=\derk_{\omega}(\Phi)\Psi+\Phi \derk_{\omega}(\Psi).
\end{eqnarray*}
By Lemma~\mref{lem:allint}, we have
\begin{eqnarray*}
\derk_{\omega} (\expcum \Phi) ={\omega} (\expcum \Phi)'
={\omega} \Phi,
\end{eqnarray*}
and so $\derk_{\omega}  \jint={\omega} \, \id_{\ratvir}$.

\mref{it:dpk2}
By Corollary~\mref{coro:RBV},
\begin{equation*}\mlabel{eq:rbw0}
(\expcum \Phi)(\expcum \Psi)=\expcum\Big(\Phi \expcum\Psi+ (\expcum \Phi)\Psi\Big).
\end{equation*}
Thus we have
\begin{eqnarray*}
P_{\alpha}(\Phi)P_{\beta}(\Psi)&=&\alpha \beta(\expcum \Phi)(\expcum \Psi)\\
&=&\alpha \expcum\Big(\Phi \beta\expcum\Psi\Big)+ \beta\Big((\alpha\expcum \Phi)\Psi\Big)\\
&=& P_{\alpha}(\Phi P_{\beta}(\Psi))+ P_{\beta}(P_{\alpha}(\Phi) \Psi),
\end{eqnarray*}
which completes the proof.
\end{proof}

\begin{remark}
In the setting of Theorem~\mref{thm:dir0}, the matching Rota-Baxter operator $P_{\omega}$ is not a
right inverse of the differential operator $\derk_{\omega}$ since we have
 $$\derk_{\omega}(P_{\omega}(\Phi)) = \derk_{\omega}({\omega}) \expcum \Phi +
 \omega \, \derk_{\omega}(\expcum \Phi) = {\omega}^2 \, \Phi\neq \Phi.$$
 Hence, for any $\omega\in\Omega$, the triple $(\ratvir, \derk_{\omega}, P_{\omega})$ is neither a
 differential Rota-Baxter ring nor an integro-differential ring for $K_\omega\neq\pm1$. On the other
 hand, $(\ratvir, \derk_{\pm1}, P_{\pm1})$ is indeed an integro-differential ring.
\end{remark}

\subsection{The integro-differential ring on linear species}
\mlabel{ss:linsp}

In this subsection, we will equip virtual linear species with an integro-differential ring
structure. Let us first recall some basic concepts of linear species~\mcite{BLL,J0}.

A \name{totally ordered set} is a pair $\ell = (U,\leq)$, where $U$ is a finite set and $\leq$ is a total order relation on $U$. Define $|\ell|:= |U|$ to be the cardinality of $\ell$.
For $u,v\in  \ell$, we write $u\leq v$ or $u\leq_\ell v$ when $u$ is less than or equal to $v$ with respect to
the order of $\ell$ and $u<v$ when $u\leq v$ and $u\neq v$. The smallest element of a nonempty totally ordered
set $ \ell$ is denoted by $m_\ell$, that is,
$$m_\ell:= \min\{x \mid x\in \ell\}.$$
Let $\ell =(U, \leq)$ be a totally ordered set. For subsets $U_1, \dots, U_k$ of $U$ such that $U_1+\cdots+U_k = U$, we write
$$ \ell_1+\cdots +\ell_k = \ell$$
if $\ell_i$ is the restriction $\ell|_{U_i}$ of $\ell$ to $U_i$, for $i=1, \ldots, k$.

An increasing function $f: \ell_1\to \ell_2$ between ordered sets $\ell_1$ and $\ell_2$ is a function such that, for $u$ and $v$ in $\ell_1$,
$$u<_{\ell_1} v \Rightarrow f(u) <_{\ell_2} f(v).$$
For $n \geq 1$, we denote by $[n]$, the ordered set $\{1,2, \ldots, n\}$ with the usual order, with the convention that $[0] = \emptyset$.  Note that for any ordered set $\ell$ of cardinality $n$, there exists a unique
increasing bijection $f: [n] \to \ell$.

The \name{ordinal sum} of two ordered sets $\ell_1 = (U_1, \leq_1)$ and $\ell_2 = (U_2, \leq_2)$, denoted by $\ell = \ell_1\osum\ell_2$, is the ordered set $\ell=(U, \leq)$,  where $U = U_1+U_2$, and, for $u,v\in U$, we have that
\[
u\leq_\ell v \Leftrightarrow
\left\{ \begin{array}{ll} u\leq_{\ell_1} v, & \text{ when } u,v\in U_1 \\u\in U_1 \text{ and } v\in U_2, \\ u\leq_{\ell_2} v, & \text{ when } u,v\in U_2.
\end{array} \right.
\]
In other words, $\ell_1 = \ell|_{U_1}$, $\ell_2 = \ell|_{U_2}$ and all the elements of $\ell_1$ are smaller than those of $\ell_2$. If $\ell_1$ and $\ell_2$ are totally ordered sets, then $\ell_1\osum\ell_2$ also is. In particular, denote by $1\osum \ell$, the totally ordered set obtained by adding a new minimum element to $\ell$.

\begin{defn}
  A \name{linear species} is a rule $F$ which
  \begin{enumerate}
  \item to each totally ordered set $\ell$, associates a finite set $F[\ell]$,

  \item to each increasing bijection $\sigma:\ell_1 \to \ell_2$, associates a function
    $ F[\sigma]: F[\ell_1] \to F[\ell_2]$.
  \end{enumerate}
  These functions $F[\sigma]$ must satisfy the following functorial properties:
  \begin{equation}
    F[\id_\ell] = \id_{F[\ell]}, \quad F[\sigma\circ \tau] = F[\sigma] \circ F[\tau].
    \mlabel{eq:spefu}
  \end{equation}
  In other words, an $\LL$-species is a functor $\LL \to \EE$, as noted in the Introduction.
\end{defn}

An element $s \in F[\ell]$ is called an \name{$F$-structure} on $\ell$, and the function $F[\sigma]$
is the \name{transport of $F$-structures} along $\sigma$. Note that ~(\mref{eq:spefu}) implies
that every transport function $F[\sigma]$ is a bijection.

\begin{defn}
  An isomorphism of $\LL$-species $\varphi: F\to G$ is a family of bijection $$\varphi_\ell: F[\ell] \to G[\ell]$$
  for each totally ordered set $\ell$, such that for any increasing bijection $\sigma:\ell_1\to \ell_2$ between totally ordered sets $\ell_1$ and $\ell_2$, one has
  $G[\sigma]\circ \varphi_{\ell_1} = \varphi_{\ell_2} \circ F[\sigma]$:
  \[
    \xymatrix{
      F[\ell_1]\ar^{\varphi_{\ell_1}}[r] \ar_{F[\sigma]}[d]&G[\ell_1]\ar_{G[\sigma]}[d]\\
      F[\ell_2]\ar^{\varphi_{\ell_2}}[r] &G[\ell_2]
    }
  \]
\end{defn}

As in the set species case, this means we have a natural isomorphism, and we write this again as
$F = G$. A number of operations on linear species can be defined.

\begin{defn}
  Let $F, G$ be linear species and $\ell$ a totally ordered set. Then we define
  \begin{enumerate}
  \item the sum, $F+G$:
    $$(F+G)[\ell]:= F[\ell] + G[\ell] \quad \text{as disjoint union};$$
  \item the product, $FG$:
    $$(FG)[\ell]:= \sum_{\ell_1+\ell_2 =\ell} F[\ell_1] \times G[\ell_2];$$
  \item the derivative, $\D F = \der{F}$:
    $$\der{F}[\ell]:= F[1\osum \ell];$$

  \item the integral, $\lint F$:
    \[(\lint F)[\ell] := \left\{\begin{array}{ll}
                                  \emptyset, & \text{if } \ell= \emptyset,\\
                                  F[\ell\setminus m_\ell],& \text{otherwise;}
                                \end{array}\right. \mlabel{d}  \]
                          \item the substitution or (partitional) composition, $F \circ G = F(G)$ (when $G(0)=0$)~\cite[(5.1.20)]{BLL}:
 $$(F\circ G)[\ell]:=\sum_{\pi\in \mathrm{Par}(\ell)} F[\ell_\pi] \times \prod_{p\in \pi} G[\ell|_p].$$
Here $\ell_\pi$ denotes the totally order set on the partition $\pi$ of $\ell$ with order induced from $\ell$, and $\ell_p$ denotes the totally ordered set on the subset $p$ of $\ell$ with order restricted from $\ell$.  
\end{enumerate}
  In each case, the transport of structures is defined in the obvious way.
  \mlabel{ldef}
\end{defn}

Note that---unlike for set species---we have a canonical choice of integral operator for linear
species. Figure~\mref{fig5} shows a graphical illustration of an $\lint F$-structure with
${\rm min} = \min \ell$.

\begin{figure}[h]
  \begin{center}
    \includegraphics[scale=0.6]{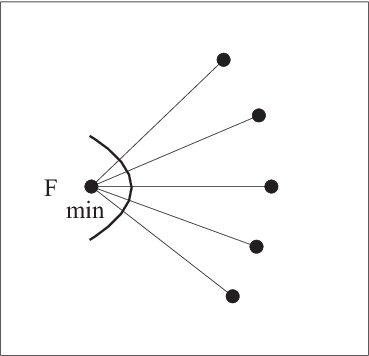}\caption{An $\cum F$-structure}\mlabel{fig5}
  \end{center}
\end{figure}

Similar to the set species in Section~\mref{sec:idrsp}, the linear species (resp.\@ virtual linear
species) also constitute a semiring (resp.\@ ring) with respect to the sum and product operations.
According to~\cite[Exercise~5.1.3(a)]{BLL}, two (virtual) linear species are isomorphic if and only if
their generating series are equal. We are thus free to denote the semiring of linear species
by~$\lspe$ and the ring of virtual linear species by~$\lvir$. This is further justified since all
the species operations of Definition~\mref{ldef} correspond exactly to the analogous series operations
by~\cite[Proposition~5.1.12]{BLL}.

The injection $\lspe \hookrightarrow \lvir, \ F\mapsto F-0$ is a semiring homomorphism. As in the
set species case, every virtual linear species $\Phi$ can be written uniquely as the reduced form
$\Phi = \Phi^+ - \Phi^-$, and the operations $\D$ and $\lint$, as well as the projections $\E$ and
$\J$ defined as in~\meqref{eq:jd}, can be extended to $\lvir$ by additivity.

Since linear species---unlike their set counterpart---enjoy a canonical integral operator, it is
hardly suprising that its associated evaluation is indeed~\cite[(5.1.35)]{BLL} evaluation at the
zero species---unlike for set species, where we have seen in Example~\mref{ex:expcum}\mref{it:nozeroin}
that this is not the case there.

\begin{lemma}
  For any linear species $F$, we have $\big(\lint F\big)' = F$ and $\E(F) = F(0)$.
  \mlabel{lem:jaction}
\end{lemma}

This can now be transferred to the virtual linear species, where we also note a few simple
consequences.

\begin{lemma}
  Let $\Phi\in \lvir$ be a virtual linear species.
  \begin{enumerate}
  \item Then we have~$(\lint \Phi)' = \Phi$ as well as $\E(\Phi) = \Phi(0)$ and
    $\J(\Phi) =\Phi - \Phi(0)$.
    \mlabel{it:ejl}

  \item For each $n\geq 0$ and each totally ordered set $\ell$, $\nder{\Phi}{n}[\ell] = \Phi[[n]\osum \ell]$.
    \mlabel{it:nord0}

  \item For each $n\geq 0$,  $\E(\nder{\Phi}{n}) = 0$ if and only if $\Phi[n] = \emptyset$.
    \mlabel{it:eton}

  \item If $\E(\nder{\Phi}{n}) = 0$ for all $n\geq 0$, then $\Phi = 0$.
    \mlabel{it:metric}
  \end{enumerate}
  \mlabel{lem:nder0}
\end{lemma}

\begin{proof}
  \mref{it:ejl} From Lemma~\mref{lem:jaction} we have that
  $$(\D \lint)(\Phi) = (\D \lint)(\Phi^+ - \Phi^-) = (\D \lint)(\Phi^+) - (\D \lint)(\Phi^-) = \Phi^+ - \Phi^- = \Phi$$
  and
  $$\E(\Phi) = \E(\Phi^+ - \Phi^-) = \E(\Phi^+) - \E(\Phi^-) = \Phi^+(0) - \Phi^-(0)= (\Phi^+ - \Phi^-)(0) = \Phi(0),$$
  and now the rest follows from $\J = \id_{\lvir}-\E$.

  \mref{it:nord0}
  We prove the result by induction on $n\geq 0$. If $n=0$, we have that
  $$\nder{\Phi}{0}[\ell] = \Phi[\ell]= \Phi[\emptyset \osum \ell] =  \Phi[[0] \osum \ell].$$
  Assume the result is true for $0\leq n\leq k$ and consider the case $n = k+1$, where we obtain
  $$ \nder{\Phi}{k+1}[\ell] = \der{(\nder{\Phi}{k})}[\ell] =\nder{\Phi}{k}[1\osum \ell]
  = \Phi[[k] \osum (1\osum \ell)] = \Phi[[k+1] \osum  \ell].$$

  \mref{it:eton} It follows from Items~\mref{it:ejl} and~\mref{it:nord0} that
  \begin{align*}
    \E(\nder{\Phi}{n})[\ell] =  \nder{\Phi}{n}(0)[\ell] = \nder{\Phi}{n} [\emptyset] = \Phi[[n] \osum \emptyset]
    = \Phi[n]
  \end{align*}
  for any totally ordered set $\ell$. Hence $\E(\nder{\Phi}{n}) = 0$ if and only if $\Phi[n] = \emptyset$.

  \mref{it:metric} This follows from Item~\mref{it:eton}.
\end{proof}

Just as the virtual set species (Theorem~\mref{thm:intd}), also linear species form an
integro-differential ring. But this crucial result is much easier to prove for the case of linear
species, due to their afore-mentioned proximity to the (integro-differential) ring of formal power
series. We present here a short proof via equivalent characterizations of integration by parts
(Lemma~\ref{lem:eqid}).

\begin{theorem}
	The triple $(\lvir, \partial, \lint)$ is an integro-differential ring.
	\mlabel{thm:idls}
\end{theorem}

\begin{proof}
	The section axiom $\partial  \lint = \id_{\lvir}$ follows from Lemma~\mref{lem:nder0}~\mref{it:ejl}.
	By Lemma~\mref{lem:eqid}, it suffices to show
	$$\E(\Phi \Psi)  =  \E(\Phi)\E(\Psi) \quad \text{ for all}\, \Phi, \Psi \in \lvir.$$
	Further by Lemma~\mref{lem:nder0}~\mref{it:ejl}, we have
	\begin{align*}
		\E(\Phi \Psi) = (\Phi \Psi)(0) =\Phi(0)\Psi(0)=\E(\Phi)\E(\Psi).
	\end{align*}
	This completes the proof.
\end{proof}

\subsection{The modified integro-differential ring on localized set species}
\mlabel{ss:locsp}

This subsection is devoted to the establishment of a unit-modified differential Reynolds ring
structure on localizations of virtual species. Let us begin with some notions and a lemma from~\mcite{GGL2}.

\begin{defn}
\begin{enumerate}
\item
A {\bf modified differential ring}\footnote{When $\lambda$ is a scalar, the operator is one of the differential type operators in~\mcite{GSZ} and the term modified differential operator is used in~\mcite{PZGL}. } $(R,D,\lambda)$ of weight $\lambda\in R$ is a pointed ring $(R,\lambda)$ with an additive operator $D:R\to R$, called a {\bf modified differential operator}, satisfying
		\begin{equation*}\mlabel{wda}
			D(xy)=D(x)y+xD(y)-x\lambda y\quad \text{ for all }\,   x,y\in R.
		\end{equation*}		
\item
A {\bf unit-modified differential ring} $(R,D)$  is a unital modified differential ring $(R,D,\lambda)$ with $\lambda=D(1_R)$. So $D$ satisfies
\begin{equation*}\mlabel{Dda}
D(xy)=D(x)y+xD(y)-xD(1_R)y\quad \text{ for all }\,   x,y\in R.
\end{equation*}
\item
A {\bf differential Reynolds  ring} $(R, D, P)$  is a unit-modified differential ring $(R, D)$ with an additive operator $P:R\to R$ satisfying
\begin{equation*}
\begin{aligned}
P(x)P(y)=&\ P(P(x)y)+P(xP(y))-P(P(x)D(1_R)P(y)), \\
D   P=&\ \id_R\quad \text{ for all }\,  x,y\in R.
\end{aligned}
\mlabel{Dra}	
\end{equation*}	

\item
A {\bf modified  integro-differential ring} $(R, D, P)$  is a unit-modified differential ring $(R, D)$ with an additive operator $P:R\to R$, called an {\bf integration}, satisfying $D   P=\ \id_R$ and
 the {\bf modified integro-differential identity}
\begin{equation*}
\begin{aligned}
\J(x)\J(y)=&\ \J(x)y+x\J(y)-\J(xy)-\J(1_R)\big(xy-\J(xy)\big),  x,y\in R.
\end{aligned}
\mlabel{mdr}	
\end{equation*}	
where $\J:=\J_{P,D}:=P  D.$
\end{enumerate}
\mlabel{de:rrb}
\end{defn}

\begin{lemma}\mcite{GGL2}
  Let $(R, D, P)$ be a unital integro-differential commutative ring of weight $0$ with an invertible
  element $\lambda$. Then $(R, D\lambda^{-1}, \lambda P)$ is a modified integro-differential ring,
  and also a differential Reynolds ring.  \mlabel{lem:uid}
  \mlabel{lem:modintd}
\end{lemma}

\begin{remark}
  As usual, the differential operator $\partial$ on $\ratvir$ can be extended to $S^{-1}\ratvir$ by
  the Leibniz rule, that is, for any $\frac{\Phi}{s}\in S^{-1}\ratvir$
  \begin{equation*}
    \bigg(\frac{\Phi}{s}\bigg)' :=\frac{s\Phi'-s'\Phi}{s^2}.
    \mlabel{eq:locder}
  \end{equation*}

\mlabel{re:locder}
\end{remark}

We are ready to state a main result for this subsection: While modifying the derivation by
\emph{pre}multiplication and the Rota-Baxter operator by \emph{post}multiplication will not produce
an integro-differential (nor a differential Rota-Baxter) structure, it is noteworthy that the idea
of Lemma~\ref{lem:modintd} does lead to a modified integro-differential (and a modified differential
Reynolds) structure.

\begin{theorem}
Let $K\in \ratvir$ and $S$ be the multiplicative subset of $\ratvir$ generated by $K$.
Define the additive operators $\dlk$ and $\ilk$ on the localization $S^{-1}\ratvir$ by
\begin{equation}
{\small
\begin{split}
\dlk: S^{-1}\ratvir \to& S^{-1}\ratvir,\quad \frac{\Phi}{s} \mapsto  \bigg(\frac{\Phi}{Ks}\bigg)'=\frac{Ks\Phi'-(Ks)'\Phi}{(Ks)^2},\\
\ilk:S^{-1}\ratvir \to& S^{-1}\ratvir,\quad \frac{\Phi}{s}  \mapsto K\expcum \frac{\Phi}{s}  :=K \sum_{i\geq 1} (-1)^{i-1}\frac{X^i}{i!} (\frac{\Phi}{s})^{(i-1)} .
\end{split}
}
\label{eq:sderx}
\end{equation}
Then $(S^{-1}\ratvir, \dlk,\ilk)$ is a modified integro-differential ring and also a differential Reynolds ring.
\mlabel{thm:dir1}
\end{theorem}

\begin{proof}
Similar to the proof of Theorem~\mref{thm:intd}, the triple $(S^{-1}\ratvir, \D, \cum_{e^X})$ is an integro-differential ring.
Then the result follows from Lemma~\mref{lem:uid}.
\end{proof}

\begin{prop}\mlabel{pp:mintdh}
With the setting of Theorem~\mref{thm:dir1}, the map sending (localized) virtual species to their generating series
$$(S^{-1}\ratvir, \dlk,\ilk) \rightarrow (\QQ[[x]], \frac{d}{dx}K(x)^{-1}, K(x)\int_0^x )$$ is a
homomorphism of modified integro-differential rings.
\end{prop}

\begin{proof}
  By Example~\mref{ex:gs} and Proposition~\mref{prop:homo}, we have
  $$\D(K^{-1}\Phi)(x)=\frac{d}{dx}\Big(K^{-1}(x)\Phi(x)\Big), \quad(K\int_{e^X}\Phi)(x)=K(x)(\int_{e^X}\Phi)(x)=K(x)\int_0^x\Phi(x)dx.$$
  Further by Lemma~\mref{lem:uid}, the formal power series ring
  $ (\QQ[[x]], \frac{d}{dx}K(x)^{-1}, K(x)\int_0^x )$ is a modified integro-differential ring.  Thus
  taking generating series is indeed a modified integro-differential ring morphism.
\end{proof}

If $K$ is a differential constant, then Theorem~\mref{thm:dir1} is reduced to the following result
leading back to the integro-differential category.

\begin{prop}
  Let $K$ be a differential constant and $S$ be the multiplicative subset of $\ratvir$ generated by
  $K$.  Then $( S^{-1}\ratvir, \dlk, \ilk)$ is an integro-differential ring with evaluation {
        $$\E\Big(\frac{\Phi}{s}\Big) =\sum_{n\geq 0} (-1)^{n}\frac{X^n}{n!} \Big(\frac{\Phi}{s}\Big)^{(n)}   \quad \text{ for all}\,
        \frac{\Phi}{s}\in S^{-1}\ratvir.$$}
  \mlabel{prop:idr}
\end{prop}

\begin{proof}
  We denote $\dlk$ and $\ilk$ simply by $\derk_{K^{-1}}$ and $P_{K}$, respectively.  Since
  $\derk_{K^{-1}}(1)=0$, the operator $\derk_{K^{-1}}$ is a derivation of weight $0$.  Notice that
  $\derk_{K^{-1}}  P_K=\id_{S^{-1}\ratvir}$ and
  $\E :=\id_{S^{-1}\ratvir}-P_K  \derk_{K^{-1}}$.  By~(\mref{eq:sderx}), we have
{\small
\begin{align*}
  \E\Big(\frac{\Phi}{s}\Big) =~& \frac{\Phi}{s}   - P_K  \derk_{K^{-1}}\Big(\frac{\Phi}{s}\Big)\\
  =~& \frac{\Phi}{s}   - P_K (\frac{1}{K} \Big(\frac{\Phi}{s}\Big)' )\\
  =~& \frac{\Phi}{s}   - K\expcum(\frac{1}{K} \Big(\frac{\Phi}{s}\Big)' )\\
  =~& \frac{\Phi}{s}   - \expcum \Big(\frac{\Phi}{s}\Big)' \\
  =~& \frac{\Phi}{s}   - \sum_{n\geq1}(-1)^{n-1} {\frac{X^n}{n!}} \Big(\frac{\Phi}{s}\Big)^{(n)} \\
=~& \frac{\Phi}{s}   + \sum_{n\geq 1} (-1)^{n}{\frac{X^n}{n!}} \Big(\frac{\Phi}{s}\Big)^{(n)} \\
=~& \sum_{n\geq 0} (-1)^{n}{\frac{X^n}{n!}} \Big(\frac{\Phi}{s}\Big)^{(n)}   \quad \text{ for all}\, \frac{\Phi}{s}\in S^{-1}\ratvir,
\end{align*}
}
which is indeed analgous to~\meqref{eq:eval}. Consequently,
{\small
\begin{align*}
  \E(\frac{\Phi}{s}\frac{\Psi}{t}) =&~\sum_{n\geq 0} (-1)^{n}{\frac{X^n}{n!}} (\frac{\Phi}{s}\frac{\Psi}{t})^{(n)} = \sum_{n\geq 0} (-1)^{n}{\frac{X^n}{n!}} \sum_{i_1+i_2=n\atop i_1, i_2\geq 0} {n\choose i_1}\Big(\frac{\Phi}{s}\Big)^{(i_1)}( \frac{\Psi}{t})^{(i_2)} \\
  =&~\sum_{n\geq 0} (-1)^{n} \sum_{i_1+i_2=n\atop i_1, i_2\geq 0} {i_1+i_2 \choose i_1} {\frac{X^n}{n!}}\Big(\frac{\Phi}{s}\Big)^{(i_1)} (\frac{\Psi}{t})^{(i_2)} \\
  =&~\sum_{n\geq 0} (-1)^{n} \sum_{i_1+i_2=n\atop i_1, i_2\geq 0}{\frac{X^{i_1}}{{i_1}!}}{\frac{X^{i_1}}{{i_2}!}}\Big(\frac{\Phi}{s}\Big)^{(i_1)} (\frac{\Psi}{t})^{(i_2)}   \qquad (\text{by ~\mref{eq:binom}} ) \\
  =&~ \sum_{i_1,i_2\geq 0}  (-1)^{i_1+i_2} {\frac{X^{i_1}}{{i_1}!}}{\frac{X^{i_1}}{{i_2}!}}\Big(\frac{\Phi}{s}\Big)^{(i_1)} (\frac{\Psi}{t})^{(i_2)}\\
  =&~ \left(\sum_{i_1\geq 0}(-1)^{i_1} {\frac{X^{i_1}}{{i_1}!}} \Big(\frac{\Phi}{s}\Big)^{(i_1)} \right) \left( \sum_{i_2\geq 0}  (-1)^{i_2}{\frac{X^{i_1}}{{i_2}!}}(\frac{\Psi}{t})^{(i_2)} \right) \\
  =&~ \E\Big(\frac{\Phi}{s}\Big) \E(\frac{\Psi}{t}).
\end{align*}
}
Hence by Lemma~\mref{lem:eqid}~\mref{it:exy}, $(S^{-1}\ratvir, \derk_{K^{-1}}, P_K)$ is an
integro-differential ring.
\end{proof}

\begin{coro}
  Let $K$ be a differential constant and $S$ be the multiplicative subset of $\ratvir$ generated by
  $K$. Then $( S^{-1}\ratvir, \dlk, \ilk)$ is a differential Rota-Baxter ring.
\end{coro}

\section{Derived structures on species from integro-differential rings}
\mlabel{s:app}

Unless otherwise specified, in this section, $(\vs, \D, \vint)$ denotes an ``integro-differential
ring of species'': either virtual species as in Theorem~\mref{thm:intd} or localized virtual species
as in Proposition~\mref{prop:idr} or virtual linear species as in Theorem~\mref{thm:idls}. Applying
the general results on integro-differential algebras from~\mcite{GKR} to these cases, we obtain a
topology and new operations on $(\vs, \D, \vint)$.

\subsection{Topology on species}
In this subsection, we turn the integro-differential ring $(\vs, \D, \vint)$ into a topological integro-differential ring.
For the topological terminology and notations not defined in this section, the author is referred to~\mcite{AP, Bour}.
Recall that $\E = \id_\vs - \vint  \D$.

\begin{defn}
	For $\Phi\in\vs$, we define {\bf the order of $\Phi$} to be $\ord(\Phi):= \min\{ n\in \NN \mid \E(\nder{\Phi}{n} )  \neq 0\}$, with the convention that $\min\emptyset = \infty$.
	\mlabel{defn:ord}
\end{defn}

\noindent Note that $\ord(0) = \infty$. For the following concepts see for example~\mcite{Bour}.

\begin{defn}
  A {\bf{filter}} on a set $\frakX$ is a set $\frakF$ of subsets of $\frakX$ which has the
  following properties:
  \begin{enumerate}
  \item Upward closed: Every subset of $\frakX$ containing a set of $\frakF$ belongs to $\frakF$.
    \mlabel{it:fil1}

  \item Downward directed: Every finite intersection of sets of $\frakF$ belongs to $\frakF$. \mlabel{it:fil2}

  \item Nontrivial: The empty set is not in $\frakF$.   \mlabel{it:fil3}
  \end{enumerate}
	\mlabel{defn:filter}
\end{defn}

\begin{defn}
  A {\bf{uniformity}} on a set $\frakX$ is a
  filter $\mathfrak{U}$ on $\frakX \times \frakX$ satisfying the following axioms:
  \begin{enumerate}
  \item Every set belonging to $\frakU$ contains the diagonal $\Delta=\{(x, x) \mid x\in \frakX\}$.

  \item If $V\in \frakU$, then $\inv{V}:=\{(x,y) \mid (y, x) \in V\} \in \frakU$.

  \item For each $V\in \frakU$, there exists $W \in \frakU$ such that $W \circ W \subseteq V,$ where
    $$W \circ W:=\{ (x,z) \mid \text{ there is } y\in \frakX \text{ such that } (x, y), (y,z)\in W\}.$$
  \end{enumerate}
  The sets of $\frakU$ are called {\bf{entourages}} of the uniformity defined on $\frakX$ by $\frakU$.
  A set endowed with a uniformity is called a {\bf{uniform space}}.
\end{defn}

\noindent The following two results are taken from~\cite[Theorem~2.16, Remark~2.18]{GKR}
and~\cite[Proposition~2.20]{GKR}.

\begin{prop}
  \begin{enumerate} \item     \label{it:unif}
  	The function
    $$d: \vs \times \vs \to \RR, \quad (\Phi, \Psi) \mapsto \dist{\Phi, \Psi} := 2^{-\ord(\Phi - \Psi)}$$
    is a pseudometric on $\vs$, inducing a uniformity on $\vs$ via the fundamental system of entourages $\frakB:= \{ B_r \mid r\in \RR_{> 0}\}$, where
    \begin{equation*}
      \mlabel{eq:bent}
      B_r := \inv{d}([0,r])= \{ (\Phi, \Psi)\in \vs\times \vs \,\mid\, \dist{\Phi,\Psi} \leq r \}.
    \end{equation*}
  \item  The uniformity on $\vs$ given in Item~\mref{it:unif} induces a unique topology on $\vs$ such that, for each $\Phi\in \vs$,
    $$\mathfrak{N}(\Phi):= \{ B_r(\Phi) \mid r\in \RR_{>0}\}$$ is the neighbourhood filter of $\Phi$
    in this topology, where $B_r(\Phi):=\{\Psi\in R \mid (\Phi, \Psi)\in B_r\}$.
  \end{enumerate}
  \mlabel{pp:ptopo}
\end{prop}

We will use this topology throughout the remainder of this section.

\begin{theorem}
  The $(\vs, \D, \vint)$ is a topological integro-differential ring. In other words, the operations
  of sum, subtraction, multiplication, derivation and integration are continuous.
  \mlabel{thm:cont4}
\end{theorem}

We are in a position to show that the space of virtual linear species is a complete metric
space. For this, we need the following lemma.

\begin{lemma}
	\mlabel{lem:dcont}
	Let $\Phi, \Psi\in \lvir$ and $n\geq 0$. Then
	$$\dist{\Phi, \Psi} \leq 2^{-(n+1)}\,\text{ if and only if }\, \Phi[i] = \Psi[i] \quad\text{ for } i=0, \ldots, n.$$
\end{lemma}
\begin{proof}
	It follows from Lemma~\mref{lem:nder0}~\mref{it:eton} that
	\begin{align*}
		&\dist{\Phi, \Psi} = 2^{-\ord(\Phi - \Psi)} \leq 2^{-(n+1)} \Leftrightarrow
		\ord(\Phi - \Psi) \geq n+1 \\
		\Leftrightarrow\quad & \E(\nder{(\Phi - \Psi)}{i}) = 0  \Leftrightarrow (\Phi - \Psi)[i] = \emptyset
		\Leftrightarrow \Phi[i] = \Psi[i] \quad\text{ for } i=0, \ldots, n,
	\end{align*}
	as required.
\end{proof}

\begin{prop}
	Let $(\lvir, \D, \vint)$ be the integro-differential ring of virtual linear species. Then
	\begin{enumerate}
		\item The function
		$d: \lvir \times \lvir \to\RR,\, (\Phi, \Psi) \mapsto \dist{\Phi, \Psi}= 2^{-\ord(\Phi - \Psi)}$
		is a metric. \mlabel{it:lmetric}
		
		\item The metric space $(\lvir, d)$ is complete. \mlabel{it:lcomp}
	\end{enumerate}
	\mlabel{pp:lcomp}
\end{prop}

\begin{proof}
	\mref{it:lmetric} Let $\Phi, \Psi\in \lvir$ with $\dist{\Phi, \Psi} = 0$. Then
	$$\ord(\Phi - \Psi) = \infty\,\text{ and so } \E(\nder{(\Phi - \Psi)}{n}) = 0\quad\text{ for all } \, n\geq 0.$$
	Hence $\Phi = \Psi$ by Lemma~\mref{lem:nder0}~\mref{it:metric}.
	Therefore by Proposition~\mref{pp:ptopo}~\mref{it:unif}, $d$ is a metric.
	
	\mref{it:lcomp} Let $(\Phi_k)_{k\geq 0}$ be a Cauchy sequence of virtual linear species.
	Then for each $n\geq 0$, there is $N_n\in \NN$ such that
	$$\dist{\Phi_p, \Phi_q} = w^{-\ord(\Phi_p - \Phi_q)} \leq 2^{-(n+1)}\quad \text{ for all}\, p, q > N_n.$$
	From Lemma~\mref{lem:dcont}, it follows that
	$\Phi_p[i] = \Phi_q[i]$ for $i = 0, \ldots, n$ and $p, q>N_n$.
	So we can define $\Phi\in \lvir$ by setting
	$$\Phi[i]:= \Phi_{N_n+1}[i] = \Phi_{N_n+2}[i] = \cdots \quad \text{ for }i = 0, \ldots, n.$$
	We claim that $\lim\limits_{k\to \infty} \Phi_k = \Phi$.
	Indeed, for any small $r>0$ with $r\in \RR$. Take $n:= \lceil \log_2(1/r)\rceil -1\in \ZZ$. From the above discussion, there is $N_n \in \ZZ_{\geq0}$ such that
	$$\Phi[i]= \Phi_{N_n+1}[i] = \Phi_{N_n+2}[i] = \cdots \quad \text{ for }i = 0, \ldots, n.$$
	So from Lemma~\mref{lem:dcont},
	$$\dist{\Phi, \Phi_k} \leq 2^{-(n+1)} = 2^{-\lceil \log_2(1/r) \rceil } \leq 2^{- \log_2(1/r)} =  r\quad \text{ for }\, k > N_n,$$
	as required.
\end{proof}

\subsection{New operations on species}
In this section, we employ the integro-differential structure to define some new operations on the
topological ring of species $(\vs, \D, \vint)$. As above, this comprises virtual (set and linear)
species as well as localized virtual (set) species.

\subsubsection{Divided powers}
The next definition and proposition are taken from~\cite[Definition~3.1]{GKR} and~\cite[Propositions~3.2, 3.3,
3.4]{GKR}. The reason for the name of divided powers is shown in Item~\mref{it:dividp} of the
proposition (the classical case being~$\divid{x}{n} = x^n/n!$).

\begin{defn}
  \mlabel{defn:divid}
  For $\Phi\in \vs$, define $\divid{\Phi}{n}\in \vs$, called {\bf the $n$-th divided power} of
  $\Phi$, inductively by
  $$\divid{\Phi}{0} = 1\,\text{ and }\, \divid{\Phi}{n} = \vint( \divid{\Phi}{n-1} \der{\Phi})\quad\text{ for }\, n\geq 1.$$
\end{defn}

\begin{prop}
  Let $\Phi, \Psi\in \vs$ and $\Omega\in \ker \E  $. Then
  \begin{enumerate}
  \item For each $n\geq 0$, the mapping $\vs \to \vs$, $\Phi\mapsto \divid{\Phi}{n}$ is  continuous.

  \item $(\divid{\Phi}{0})' = 0$ and $(\divid{\Phi}{n})' = \divid{\Phi}{n-1} \der{\Phi}$ for $n\geq 1$. \mlabel{it:dividd}

  \item $\E(\divid{\Phi}{0}) = 1$ and $\E(\divid{\Phi}{n}) = 0$ for $n\geq 1$. \mlabel{it:divide}

  \item $n! \divid{\Omega}{n} = \Omega^n$ for $n\geq 0$. \mlabel{it:dividp}

  \item $\divid{\Phi}{n} \divid{\Phi}{m} = \binc{n+m}{n} \divid{\Phi}{n+m}.$ \mlabel{it:dividpr}
  \end{enumerate}
  \mlabel{pp:dividp5}
\end{prop}

\begin{coro}\mlabel{coro:exp}
  For any $\Phi\in \vs$, we have
  $$n! \divid{\Phi}{n}=\J(\Phi)^n.$$
  Moreover, Proposition~\mref{pp:dividp5}~\mref{it:dividp} holds for $\Omega=\J(\Psi)$ with any
  $\Psi\in\vs$.
\end{coro}

\begin{proof}
  The first part follows from $\divid{\Phi}{1}=\J(\Phi)$ and
  Proposition~\ref{pp:dividp5}~\mref{it:dividpr}. For $\Psi\in\vs$, we have
  $$ n! \divid{\Omega}{n} =  n! \divid{\J(\Psi)}{n}=\J(\Phi)^n=\Omega^n,$$
  as required.
\end{proof}

\subsubsection{Composition}

We introduce some new operations on the completion $(\cvs, \D,\vint)$ of $(\vs, \D, \vint)$ whose
evaluation shall be written as $\E:= \id_{\cvs} - {\int}  \D$. Notice that the space of virtual
linear species $\lvir$ is already complete by Proposition~\mref{pp:lcomp}~\mref{it:lcomp}.

We recall from~\cite[\S2.2]{BLL} that the \name{functorial composition} $\bcom$ on set species is
defined by
\begin{equation}\mlabel{eq:fcombll}
  \Phi \bcom \Psi[U]:=\Phi[\Psi[U]] \quad{\rm\ for\ }\, \Phi, \Psi\in \vs.
\end{equation}
On the other hand, an \name{integro-differential composition} is defined in~\mcite{GKR} by
$$ \Phi \com \Psi: = \sum_{n\geq 0}  {\E}(\nder{\Phi}{n}) \divid{\Psi}{n}\quad {\rm\ for\ }\, \Phi\in \cvs,\Psi\in\ker \E.$$
It is known~\cite[(1.2.22)]{BLL} that two species $F$ and $G$ are equipotent $F\equiv G$ if and only if 
$F(x)=G(x)$. The two compositions are in fact equipotent, as we will now show. Note that the next
two results only make sense for virtual set species (in so far as~$\bcom$ is only defined for
these).

\begin{theorem}
\mlabel{pp:funcomp}
  Let $(\vs, \D, \vint)$ be the integro-differential ring of virtual set species in
  Theorem~\mref{thm:intd}.  For $\Phi\in \vs$ and $\Psi\in \vs \cap\ker \E $, we have
$$\Phi\bcom  \Psi\equiv \Phi \com \Psi.$$
\end{theorem}
\begin{proof}
We have
$$  \Phi \com \Psi=\sum_{n\geq 0}  {\E}(\nder{\Phi}{n}) \divid{\Psi}{n}=\sum_{n\geq 0} \Phi^{(n)}(0)\frac{\Psi^n}{n!}=\sum_{n\geq 0} \frac{\Phi^{(n)}(0)}{n!}\Psi^n.$$
On the one hand, the series associated to the species $\Phi\bcom \Psi$ in the sense of ~(\mref{eq:fcombll}) is
$$(\Phi\bcom \Psi)(x)\overset{\text{\cite[p.73]{BLL}}}{=}\Phi(x)\bcom \Psi(x)\overset{\text{\cite[p.71~(4)]{BLL}}}{=}\sum_{n\geq0,m:={ \Psi(x)}^{(n)}|_{x=0}}\frac{{ \Phi(x)}^{(m)}|_{x=0}}{n!}x^n=\sum_{n\geq 0} \frac{{ \Phi(x)}^{(n)}|_{x=0}}{n!}\Psi(x)^n.$$
On the other hand, the series associated to the species $\Phi \com \Psi$ is
{\small $$\Big(\sum_{n\geq 0}  {\E}(\nder{\Phi}{n}) \divid{\Psi}{n}\Big)(x)\overset{\text{\cite[p.33]{BLL}}}{=}\sum_{n\geq 0} \frac{\Phi^{(n)}(0)(x)}{n!}\Psi^n(x)=\sum_{n\geq 0} \frac{\frac{d^n \Phi(x)}{d^nx}|_{x=0}}{n!}\Psi(x)^n=\sum_{n\geq 0} \frac{{ \Phi(x)}^{(n)}|_{x=0}}{n!}\Psi(x)^n .$$}
Thus
$(\Phi\bcom \Psi)(x)=(\Phi \com \Psi)(x)$, and so $\Phi\bcom  \Psi\equiv \Phi \com \Psi.$
\end{proof}

The following properties from~\cite[Propositions~4.5, 4.6]{GKR}) hold for any integro-differential algebra
but, as pointed out above, their usage in species theory appears to be confined to virtual set
species.

\begin{prop}
  \mlabel{pp:comp5}
  Let $\Phi,\Psi\in \cvs$ and $\Omega\in \ker  {\E}$. Then
  \begin{enumerate}
  \item $\E(\Phi\com \Omega) = \E(\Phi)$. \mlabel{it:come1}

  \item  $(\Phi+\Psi) \com \Omega = \Phi\com \Omega + \Psi\com \Omega$. \mlabel{it:comsdis}

  \item $(\Phi \Psi)\com \Omega = (\Phi\com \Omega)(\Psi\com \Omega)$. \mlabel{it:compdis}

  \item $\Phi\com \Omega = \Phi$ if $\Phi\in \ker \dd$. \mlabel{it:compker}

  \item $\der{(\Phi\com \Omega)} = (\der{\Phi} \com \Omega)\, \der{\Omega}$. \quad (chain rule)\mlabel{it:chainc}

  \item $\pp((\Phi\com \Omega)\, \der{\Omega}) = \pp(\Phi) \com \Omega$. \quad (substitution rule) \mlabel{it:subc}

  \item $\divid{(\Phi\com \Omega)}{n} = \divid{\Phi}{n} \com \Omega$. \mlabel{it:compdiv}

  \item $(\Phi\com \Psi)\com \Omega = \Phi\com (\Psi\com \Omega)$ if $\Psi\in \ker  {\E}$. \mlabel{it:compass}
  \end{enumerate}
\end{prop}

Denote by $\cvs^\ast$ the group of invertible elements of $\cvs$.  The following maps introduced
in~\cite[Definitions 5.2, 5.6]{GKR} along with their properties from~\cite[Propositions~5.3, 5.7]{GKR} can be
established for any integro-differential algebra (and thus for all three species rings).

\begin{defn}
  \mlabel{defn:exp}
  Let $\Phi, \Psi\in \cvs$. We define
  \begin{align*}
    \exp:&\, \ker {\E} \to 1+ \ker {\E},\quad \Phi \mapsto \sum_{n\geq 0} \divid{\Phi}{n},\,\text{ and }\\
    \log:&\, \cvs^\ast \to \ker {\E}, \quad \Phi\mapsto  \vint \inv{\Phi} \der{\Phi}.
  \end{align*}
\end{defn}

\begin{prop}
  Let $\Phi\in \ker {\E}$ and $\Psi\in \cvs^\ast$.
  Then
  \begin{enumerate}
  \item The functions $\exp$ and $\log$ are continuous. \mlabel{it:expc}

  \item $\der{(\exp(\Phi))} = \exp(\Phi) \der{\Phi}$ and
    $\der{(\log (\Psi))} = \inv{\Psi} \der{\Psi}$. \mlabel{it:explogd}
  \end{enumerate}
  \mlabel{pp:expc}
\end{prop}

When the base ring is a field of characteristic zero, by Proposition~\mref{pp:dividp5}~\mref{it:dividp} we have
$\Phi^{[n]}=\Phi^n/n!, n\geq 0$. Thus the analytic exponential $e^\Phi$ of the species $\Phi$ coincides with $\exp(\Phi)$:
\begin{equation}
e^{\Phi}=\exp(\Phi).
\mlabel{eq:twoexp}
\end{equation}

Also, applying a general result on integro-differential algebra~\cite[Lemma 5.4]{GKR}, $(1+\ker {\E},\cdot)$ is a subgroup of $(\cvs^\ast,\cdot)$. Further applying \cite[Propositions~5.5, 5.8, 5.10]{GKR}, we obtain

\begin{prop}
  \mlabel{pp:egph}
  Let $\Phi\in \ker {\E}$ and $\Psi\in \cvs^\ast$.
  Then
  \begin{enumerate}
  \item The function $\exp\colon (\ker {\E},+) \to (1+\ker {\E},\cdot)$ is a group
    homomorphism, that is,
    $$\exp(0) = 1,\, \exp(\Phi + \Psi) = \exp(\Phi) \exp(\Psi),\, \exp(-\Psi) = \exp(\Psi)^{-1}\quad\text{ for }\, \Phi,\Psi\in \ker {\E}.$$

  \item The function $\log\colon (\cvs^\ast,\cdot) \to (\ker {\E}, +)$ is a group
    homomorphism, that is,
    $$\log(1) = 0,\, \log(\Phi\Psi) = \log(\Phi) + \log(\Psi),\, \log(\Phi^{-1}) = -\log(\Phi)\quad\text{ for }\, \Phi,\Psi\in \cvs^\ast.$$

  \item The function $\exp\colon (\ker {\E},+) \to (1+\ker {\E}, \cdot)$ and
    $\log: (1+\ker {\E}, \cdot)\to (\ker {\E},+) $ are inverse group isomorphisms.
    \mlabel{it:egph3}
  \end{enumerate}
\end{prop}

Applying ~\meqref{eq:twoexp} and Proposition~\mref{pp:egph}~\mref{it:egph3}, we conclude that
the analytic logarithm, defined in~\cite[p. 9]{Lab5} as the inverse of the analytic exponential,
agrees with the function $\log$ defined above (but not with the combinatorial logarithm
in~\cite{BLL}).

Using $\exp$ and $\log$, we can now introduce exponentiation relative to an arbitrary invertible
base such that some well-known natural properties can be proved to hold~\cite[Proposition~5.12]{GKR}.

\begin{defn}
  \mlabel{defn:expon}
  For $\Phi\in \cvs^\ast$ and $\Psi\in \cvs$ we define \name{exponentiation}
  $\Phi^\Psi \in 1+\ker {\E}$ with \name{base} $\Phi$ and \name{exponent} $\Psi$ by
  $$\Phi^\Psi:= \exp(\Psi \log (\Phi)).$$
\end{defn}

\begin{prop}
  Let $\Phi, \Phi_1, \Phi_{2}\in \cvs^\ast$ and $\Psi, \Psi_{1}, \Psi_{2}\in \cvs$. Then
  \begin{enumerate}
  \item $\Phi^{\Psi_{1} + \Psi_{2}} = \Phi^{\Psi_{1}} \Phi^{\Psi_{2}}$. \mlabel{it:esum}

  \item $(\Phi_1 \Phi_{2})^\Psi  = \Phi_1^\Psi \Phi_{2}^\Psi$. \mlabel{it:epro}

  \item $(\Phi^{\Psi_{1}})^{\Psi_{2}} = \Phi^{\Psi_{1} \Psi_{2}}$. \mlabel{it:epow}

  \item $\log (\Phi^\Psi) = \Psi \log (\Phi)$. \mlabel{it:elog}

  \item $\der{(\Phi^\Psi)} = \Phi^\Psi \der{\Psi}\log(\Phi) + gf^{\Psi-1} \der{\Phi}$. \mlabel{it:eder}

  \item $(\exp(\Omega))^\Psi = \exp(\Psi\, \Omega)$ with $\Omega\in \ker {\E}$. \mlabel{it:egpow}

  \end{enumerate}
  \mlabel{pp:exponen}
\end{prop}

\noindent
{\bf Acknowledgements}:
This work is supported by NNSFC (12071191) and Innovative Fundamental Research Group Project of Gansu Province (23JRRA684). We would like to thank Gilbert Labelle who was instrumental in drawing up the basic idea of integration by parts for species (whose iteration has led Joyal to the eponymous integral). His inspiring talk on combinatorial integration at the ACA 2014 conference at Fordham University has laid the foundation for our present work.

\noindent
{\bf Declaration of interests. } The authors have no conflicts of interest to disclose.

\noindent
{\bf Data availability. } Data sharing is not applicable as no new data were created or analyzed.

\end{document}